\documentclass[11pt,a4paper]{article}
\usepackage{amsmath,cite}
\usepackage{amssymb}
\usepackage{epsfig}
\usepackage{psfrag}
\usepackage{latexsym}
\usepackage[mathcal]{eucal}
\usepackage{lscape}
\usepackage{comment}
\usepackage{color}
\usepackage{hyperref}
\usepackage{graphicx}
\usepackage{caption}
\usepackage{subcaption} 
\usepackage{titlesec}
\usepackage{todonotes}
\usepackage{atveryend}
\usepackage{float}
\usepackage[normalem]{ulem}
\makeatletter
\let\origcitation\citation
\AtEndDocument{\def\mycites{\@gobble}%
  \def\citation#1{\g@addto@macro\mycites{,#1}\origcitation{#1}}}
\AtVeryEndDocument{\typeout{***^^JCited keys: \mycites^^J***}}
\makeatother
\newcommand*\circled[1]{\tikz[baseline=(char.base)]{
            \node[shape=circle,draw,inner sep=0.5pt] (char) {#1};}}

\usepackage{comment}
\hypersetup{
    colorlinks=true,
    linkcolor=blue,
    filecolor=magenta,
    urlcolor=cyan,
    citecolor=red,
}
\newcommand{\ds}{\displaystyle}

\newcommand{\nexto}{\kern -0.54em}
\newcommand{\R}{\mathbb{R}}

\newcommand{\proofbox}{\hspace{\fill}{$\Box$}}
\newtheorem{algorithm}{Algorithm}[section]
\newtheorem{lemma}{Lemma}[section]
\newtheorem{theorem}{Theorem}[section]
\newtheorem{corollary}{Corollary}[section]

\newtheorem{remark}{Remark}[section]

\newenvironment{proof}{Proof.}{\proofbox}
\renewcommand{\theequation}{\arabic{equation}}

\newcommand{\argmin}{\operatornamewithlimits{argmin}}

\DeclareMathOperator\sign{sign}

\begin{document}

\title{\bf A Primal--Dual Penalty Algorithm and Optimal Control of the Double Integrator}

\author{Regina S. Burachik\thanks{School of Mathematical Sciences, Adelaide University, City Campus East, North Terrace, Adelaide, SA~5005, Australia.
Email: \texttt{regina.burachik@adelaide.edu.au}, \texttt{yalcin.kaya@adelaide.edu.au}, \texttt{xuemei.liu@mymail.unisa.edu.au}.} \quad\  C. Yal{\c c}{\i}n Kaya\footnotemark[1] \quad\ Xuemei Liu\footnotemark[1]}

\date{June 4, 2026}

\maketitle

\begin{quotation}
\begin{abstract}
\noindent We investigate the behaviour of a primal--dual penalty scheme applied to all possible instances, i.e., boundary conditions, of an optimal control problem involving the double integrator. The proposed duality-based framework enables us to derive closed-form expressions for the associated dual functions as well as for the optimal control variables. Furthermore, we provide a complete characterization of the least exact penalty parameter for this problem. The iteration dynamics of the algorithm are examined by constructing and analyzing the trajectories of the constraint functions for every admissible instance of the problem. Finally, we compare the two step-size rules employed in our primal--dual algorithm with two alternative Polyak-type step-size strategies.
\end{abstract}
\end{quotation}

\begin{verse}
\end{verse}

\begin{verse} 
\end{verse}

\pagestyle{myheadings}
\markboth{}{\sf\scriptsize A Primal--Dual Penalty Algorithm and Optimal Control of the Double Integrator\ \ by R.~S.~Burachik, C.~Y.~Kaya and X.~Liu}

\section{Introduction}

The {\em double integrator} is one of the simplest ODEs, given by $\ddot y(t) = u(t)$, where $\ddot{y} := d^2 y/dt^2$ and $u(t)$ is a specified ``forcing term''.  The solution to this ODE can be obtained by integrating $u(t)$ twice, hence the name ``double integrator''.  For a mechanical system, $y(t)$, $\dot{y}(t) := dy(t)/dt$ and $\ddot{y}(t)$ represent the position,  velocity, and acceleration at time $t$, respectively, of a unit point mass in rectilinear motion on a frictionless horizontal plane.  The double integrator also models the analogous rotational-mechanical and electrical systems~\cite{Wellstead2005}, and acts as a building block for cubic splines, which are variational curves $y(t)$ that minimize the squared ${\cal L}^2$-norm of their acceleration $\ddot{y}$~\cite{Dontchev1993, OpfObe1988}.

Because of its simplicity, the double integrator has long been used to illustrate both basic and new concepts in optimal control theory and to test the performance of numerical algorithms in solving optimal control problems; see, for example,~\cite{BauBurKay2019, BurCalKayMou2024, BurCalKayMou2026, BurKayMou2024, Kaya2020}, and the research-level text book~\cite{Locatelli2017} that is solely devoted to this model.  The role of the double integrator in optimal control is arguably similar to the role of the {\em Dahlquist test problem} in numerical analysis: the simple initial value problem $\dot y(t) := \lambda\,y(t)$, $y(0) = y_0$, whose solution can be written immediately as $y(t) = y_0\,e^{\lambda\,t}$ with $\lambda$ a constant parameter, is typically used to evaluate the stability of ODE solvers (e.g., Runge--Kutta methods) and perform error analysis; see, for example,~\cite[Section~5.11]{BurFai2011} and \cite{CorKay2026, CorKayMoi2019}.

Our aim in the present paper is to study the behaviour of the {\em primal--dual penalty} (PDP) {\em algorithm} applied to the optimal control of the double integrator.  The PDP algorithm has recently been proposed in~\cite{BurKayLiu2023} (also see~\cite{Liu2022}) to solve (general) constrained infinite-dimensional optimization problems and, in particular, optimal control problems.  Although the working and efficiency of the PDP algorithm were illustrated via some challenging optimal control problems in~\cite{BurKayLiu2023}, motivated by the aforementioned practices, our objective here is to study in utmost detail the behaviour of the PDP algorithm applied to solve a test optimal control problem involving the double integrator.

Let $x_1 := y$ and $x_2 := \dot{y}$, and consider the {\em boundary conditions} $y(0) = s_0$, $y(1) = s_f$, $\dot{y}(0) = v_0$, $\dot{y}(1) = v_f$, for the double integrator over the time horizon $[0,1]$.  We aim to use the following {\em optimal control problem}, which is given in standard form, as a test problem for the PDP algorithm.
\[
\mbox{(P)}\,
  \left\{
    \begin{array}{rlll}
      \ds\min_{u(\cdot)} & \ds\frac{1}{2}\int_0^1 u^2(t)\,dt & \!\!\!=: \ds\frac{1}{2}\,\|u\|_{{\cal L}^2}^2 & \\[4mm]
      {\rm subject \, to} & \dot x_1(t)=x_2(t)\,, & x_1(0)=s_0\,, & x_1(1)=s_f\,, \\[2mm]
       & \dot x_2(t)=u(t)\,, & x_2(0)=v_0\,, & x_2(1)=v_f\,,
    \end{array}
  \right.
\]
where the functions $x_1$ and $x_2$ are called the {\em state variables} and the function $u$ the {\em control variable}.  Problem~(P) is concerned with finding the {\em optimal control} $u$ that minimizes the {\em functional} expressed as the squared ${\cal L}^2$-norm of $u$, denoted by $\|u\|_{{\cal L}^2}^2$, subject to the double integrator (dynamics) and boundary conditions.  Problem~(P) is traditionally referred to as a {\em minimum-energy control problem} in the optimal control literature 
and is a favourite example in lecture courses and text books on optimal control.  In fact, the solution to Problem~(P) can be easily written down via the calculus of variations: the resulting optimal control $u(t)$ is linear in $t$ and so the corresponding $x_1(t)$ (or $y(t)$) after integrating $u(t)$ twice is cubic in $t$.  In fact, $x_1(t)$ obtained by solving~(P) is an example of a {\em variational curve} (albeit simple) in the approximation theory and differential geometry (see Section~\ref{sec:solution} for the detailed solution).

The PDP algorithm exploits a generalized Lagrangian duality scheme for infinite-dimensional constrained optimization problems, a scheme where the classical Lagrangian is augmented by a general penalty term, as described and studied in \cite{BurachikXuemei2023, BIMinexact2013}.  More details on augmented Lagrangian duality and its properties can be found in \cite{BurachikXuemei2023, BurKayLiu2023,BIMinexact2013}, as well as in \cite{Liu2022}.  When the PDP algorithm is applied to solve Problem~(P), for which we have the analytical solution, the steps of the algorithm are likely to reveal new properties of the problem, especially those properties arising from duality (see Contributions C1--C3 below). Moreover, we can also better understand the behaviour of the algorithm itself (see Contribution C4 below).

Next, we describe the theoretical framework for studying Problem~(P).  Let $H$ be a Hilbert space. Consider the following infinite-dimensional equality-constrained optimization problem.
\begin{equation}\label{eq:generalPr}
\min_{u\in H}\ \varphi(u)\ \ {\rm  s.t. }\ \ h(u)=0\,,
\end{equation}
where $u$ is the {\em primal variable}, the functional $\varphi:H\to\R$ is lower semi-continuous, and $h:H\to H$ is some given continuous mapping.  We will show in Section~\ref{sec:reformulation} that Problem~(P) can be reformulated as the general problem expressed in~\eqref{eq:generalPr}.  

As in \cite{BurKayLiu2023}, the type of augmented Lagrangian we consider for~\eqref{eq:generalPr} is an extension of the augmented Lagrangian in \cite{BFK2014,BKdsg} to infinite dimensions. The infinite-dimensional problem in~\eqref{eq:generalPr} has recently been studied in \cite{BurachikXuemei2023}, where  optimal control problems are addressed via a primal--dual scheme and the augmented Lagrangian $\widetilde{L}:H\times H\times\R_+\to\R$ is defined for~\eqref{eq:generalPr} as follows:
\begin{equation}\label{eq:generalLagrn}
    \widetilde{L}(u,z,c):= \varphi(u)-\langle z,A(h(u))\rangle+ c\,\sigma(h(u))\,,
\end{equation}
where $z$ and $c$ are the {\em dual variables}, $A: H\to H$ is a suitable map, and $\sigma:H\to\R_+$ verifies $\sigma(v)=0$ if and only if $v=0$.  The resulting primal--dual scheme is paired in \cite{BurachikXuemei2023} with an epsilon subgradient technique that solves the dual problem.  An important advantage of the Lagrangian $ \widetilde{L}$ defined in \eqref{eq:generalLagrn} is that, unlike other available versions of the augmented Lagrangian (such as those proposed in \cite{Gdsg,BGKIdsg,BKMinexact2010,BIMdsg,BIMinexact2013, BurKay2012}), it induces a penalty function method for the choice $A=0$; also see~\cite{BurKayPri2022} for a penalty method emanating from the duality of~\eqref{eq:generalLagrn}, although it is in finite dimensions.

Indeed, we will take $A=0$ and $\sigma$ as a suitable norm in $H$. In this situation, $\sigma$ becomes the {\em penalty term function} and $c$ the {\em penalty parameter}.  Therefore, we redefine the augmented Lagrangian as $L:H\times \R_+\to\R$ such that
\begin{equation}\label{eq:pen_fun}
    L(u,c):= \varphi(u) + c\,\sigma(h(u))\,,
\end{equation}
where $c$ is the only dual variable.  The fact that the method constitutes a penalty approach justifies the term primal--dual penalty (PDP) algorithm for the resulting epsilon-subgradient method.  We thus can interpret the Lagrangian $L$ in~\eqref{eq:pen_fun} is a {\em penalty function}. 
The {\em dual function} corresponding to the Lagrangian $L$ is defined as
\begin{equation}\label{eq:dualFunction_P0}
    q(c):=\inf_{u\in H}\, L(u,c)\,,
\end{equation}
and the {\em dual problem} induced by these choices is
\begin{equation}\label{eq:dualProblem_P0}
  \max_{c\ge 0}\, q(c)\,.
\end{equation}

As mentioned above, the PDP algorithm uses an 
epsilon subgradient direction in its update rule (for the penalty parameter $c$) and improves (i.e., increases) the dual values in each iteration.  The main aim of the present paper is to further exploit this {\em duality induced} penalty update and to use the resulting method to better understand and discover new aspects of the PDP algorithm, as well as to establish the resulting duality features of the double integrator problem. The contributions of this paper can be summarized as follows.  
\begin{itemize}
\item[C1] For the primal--dual framework that we use, we first rewrite the differential equation constraints in Problem~(P) as integral constraints.  In particular, we derive and define the vector of constraint functions $h(u)$ (see~\eqref{eq:h_p1} below), and thus express the generalized augmented Lagrangian (or the penalty function) specific for Problem~(P) in~\eqref{eq:LagrnCarUnconstr}.  Subsequently, by minimizing this Lagrangian, we obtain and present closed-form expressions for the dual function in Theorem~\ref{th:dualCarUnconstrIFF}, on a case-by-case basis. The different cases arise by considering all instances of Problem~(P), classified by the possible choices of the initial data $s_0,s_f,v_0,v_f$.

\item[C2] A {\em least exact penalty parameter} is the smallest positive number $\overline{c}$ such that for every $c\ge\overline{c}$ a minimizer of the Lagrangian in~\eqref{eq:pen_fun} is a solution to Problem~(P).  In Corollary~\ref{cor:solution_exactParameter_p1}, we provide a closed form expression (given only in terms of the boundary conditions of the double integrator) for the least exact penalty parameter $\overline{c}$ for Problem~(P).

\item[C3] In Section~\ref{sec:constraints_p1}, we look at the behaviour of the constraint function and compute the dual function, for all possible choices of the initial conditions. The resulting trajectories of the constraint function are visualized in Figure~\ref{fig:subfiguresz1z2} and interpreted within various regions that are uniquely defined by the boundary conditions.  

\item[C4] In Sections~\ref{sec:iterates_p1} and \ref{sec: PolyakComparison}, we compute PDP iterates and make comparisons for different choices of the step-size used to update $c$.
\end{itemize}

\section{Preliminaries}
\label{sec:reformulation}

\subsection{A reformulation of Problem (P)}

We now reformulate (P) as an equality constrained problem, which allows us to write down the Lagrangian in \eqref{eq:pen_fun}. Indeed, using the ODE constraints and the properties of definite integrals, we can rewrite Problem~(P) as
\[
\mbox{(P1)}\,
  \left\{
    \begin{array}{rl}
      \min & {\ds\frac{1}{2}\int_0^1 u^2(t)\,dt} \\[4mm]
      {\rm subject \, to} &x_1(1)-s_f=0\,, \\[2mm]
       & x_2(1)-v_f=0\,, \\
    \end{array}
  \right.
\]
where $x_1(1)$ and $x_2(1)$ are affine functions of $u$ expressed from the ODEs in Problem~(P) as
\begin{equation} \label{eq:xP1reformulation}
\begin{array}{ll}
& \ds x_2(1)=v_0+ \int_0^1 u(\tau)\,d\tau\,, \\[5mm]
& \ds x_1(1)=s_0+\int_0^1 x_2(\tau)\,d\tau
 =s_0+\int_0^1\left[v_0+\int_0^\tau u(s)ds\right]d\tau\,.
\end{array}
\end{equation}
Next, define the constants
\begin{equation}\label{eq:r1_r2}
    r_1:= v_0+s_0-s_f\quad \hbox{and}\quad r_2:= v_0-v_f\,, 
\end{equation}
and the vector $r:=(r_1,r_2)^T\in \R^2$.  Equations~\eqref{eq:xP1reformulation} and \eqref{eq:r1_r2} suggest that the equality constraints can be defined through the vector mapping $h:{\cal L}^2([0,1];\R)\to \R^2$ as
 \begin{equation}\label{eq:h_p1}
 h(u):=  \left[ 
 \begin{array}{l}
     h_1(u)  \\[2mm]
     h_2(u) \\
 \end{array}
 \right] := \left[ 
 \begin{array}{l}
     r_1+\ds\int_0^1\left[ \int_0^\tau u(s)ds \right]d\tau   \\[5mm]
      r_2+ \ds\int_0^1 u(\tau)d\tau\\
 \end{array}
 \right].
 \end{equation}
 With this definition, it is direct to check that the constraints in Problem (P1) can be rewritten as $h(u)=0$.
Taking now $\varphi(u) := (1/2)\|u\|_{{\cal L}^2}^2$, we have reformulated (P1) as a special case of \eqref{eq:generalPr}.

The penalty function $\sigma$ is chosen as the finite-dimensional $\ell_1$-norm, denoted as $\|\cdot\|_1:=|\cdot|+|\cdot|$, where $|\cdot|$ is the {\em absolute value function}.  Therefore, the augmented Lagrangian given in~\eqref{eq:pen_fun} can then be expressed for Problem~(P1) as
\begin{equation}\label{eq:LagrnCarUnconstr}
L(u,c)=\frac{1}{2}\int_0^1u^2(t)\,dt +
c\left(\, \left| r_1+\ds\int_0^1\left[ \int_0^\tau u(s)\,ds \right]d\tau \right|
+\left|r_2+ \ds\int_0^1 u(\tau)\,d\tau\right|\, \right).
\end{equation}

\subsection{The PDP algorithm}
\label{sec:algPDP}

In what follows, we provide a description of the PDP algorithm from~\cite{BurKayLiu2023}. 

\begin{algorithm} \label{alg:PDP}
{\bf (Primal--Dual Penalty (PDP) Algorithm)}
\begin{quote}\rm
Let $\alpha,\varepsilon>0$.  Choose a sequence $\{\alpha_k\}\subset (0,\alpha)$.\\[1mm]
{\bf Step \boldmath{$0$}}. (Initialization) Choose $c_0>0$ and let $k:=0$.

\noindent {\bf Step \boldmath{$1$}}. (Solution of Subproblem and Stopping Criterion)
\begin{itemize}
    \item[(a)] Find $\ds u_k\in \argmin_{u\in H}{L(u,c_k)}$.
            \item[(b)] If $\| h(u_k)\|_\infty < \varepsilon$, stop.
\end{itemize}

\noindent {\bf Step \boldmath{$2$}}. (Selection of step-size and Update of Dual Variables)\\
Choose $s_k>0$ and set $\tilde{s}_k = (\alpha_k+1)s_k$. Update the penalty parameter by
\[
c_{k+1}:= c_k + \tilde{s}_k\,\|h(u_k)\|_1\,.
\]
Set $k:=k+1$, go to Step $1$.
\end{quote}
\end{algorithm}
The step-size $s_k$ to be used in the PDP algorithm is chosen as in~\cite[Algorithm 2]{BIMinexact2013}, which gives rise to the versions PDP-1 and PDP-2 of the PDP algorithm.
\begin{itemize}
    \item Algorithm PDP-1: Take two parameters $\beta>\eta>0$. Let $u_k$ be as in Step 1(a). Consider the step-size
\begin{equation}\label{sk_DSG1}
s_k\in [\eta_k,\beta_k]\,,
\end{equation}
where $\eta_k:=\min\{\eta,\|h(u_k)\|_2\}$ and $\beta_k:=\max\{\beta,\|h(u_k)\|_1+\|h(u_k)\|_2\}$, where $\|\cdot\|_2$ is the finite-dimensional $\ell_2$ norm. 
    \item Algorithm PDP-2: Take  $\beta>0$ and a sequence $\{\theta_k\} \subset\R_+$  such that
$\sum_j\theta_j=\infty$, and $\theta_k\leq\beta$ for all $k$.
Let $u_k$ be as in Step 1(a). Consider the step-size
\begin{equation}\label{sk_DSG2}
s_k\in [\eta_k,\beta_k]\,,
\end{equation}
where $\eta_k:= \theta_k/\|h(u_k)\|_1$ and $\beta_k:= \beta/\|h(u_k)\|_1$.  
\end{itemize}

\subsection{Analytical solution to Problem~(P)}
\label{sec:solution}

The analytical solution to Problem~(P) can simply be given as (see, for example, \cite{Kaya2020})
 \begin{equation}\label{solnCarUnconstr}
\begin{array}{rcl}
u(t)&=& c_1t+c_2\,,\\[2mm]
x_1(t) &=& \dfrac{1}{6}c_1t^3+\dfrac{1}{2}c_2t^2+v_0t+s_0\,,\\[4mm]
 x_2(t)&=& \dfrac{1}{2}c_1t^2+c_2t+v_0\,,
\end{array}
\end{equation}
for all $t\in [0,1]$, where
\begin{equation}\label{cPenaltyCarUnconstr}
\begin{array}{rcl}
c_1 &=& -12(s_f-s_0)+6(v_0+v_f)\,,\\[2mm]
c_2 &=& 6(s_f-s_0)-2(2v_0+v_f)\,.
\end{array}
\end{equation}
For the case where $s_0=0$, $s_f=0$, $v_0=1$, and $v_f=0$, the solution $u(t) = 6\,t - 4$ is shown in Figure~\ref{fig:solnUnconstr}(a) and the corresponding $x_1(t) = t^3 - 2\,t^2 + t$ and $x_2(t) = 3\,t^2 - 4\,t + 1$ in Figure~\ref{fig:solnUnconstr}(b).

 \begin{figure}
  \centering
  \begin{subfigure}{.49\textwidth}
\centering
  \includegraphics[width=\textwidth]{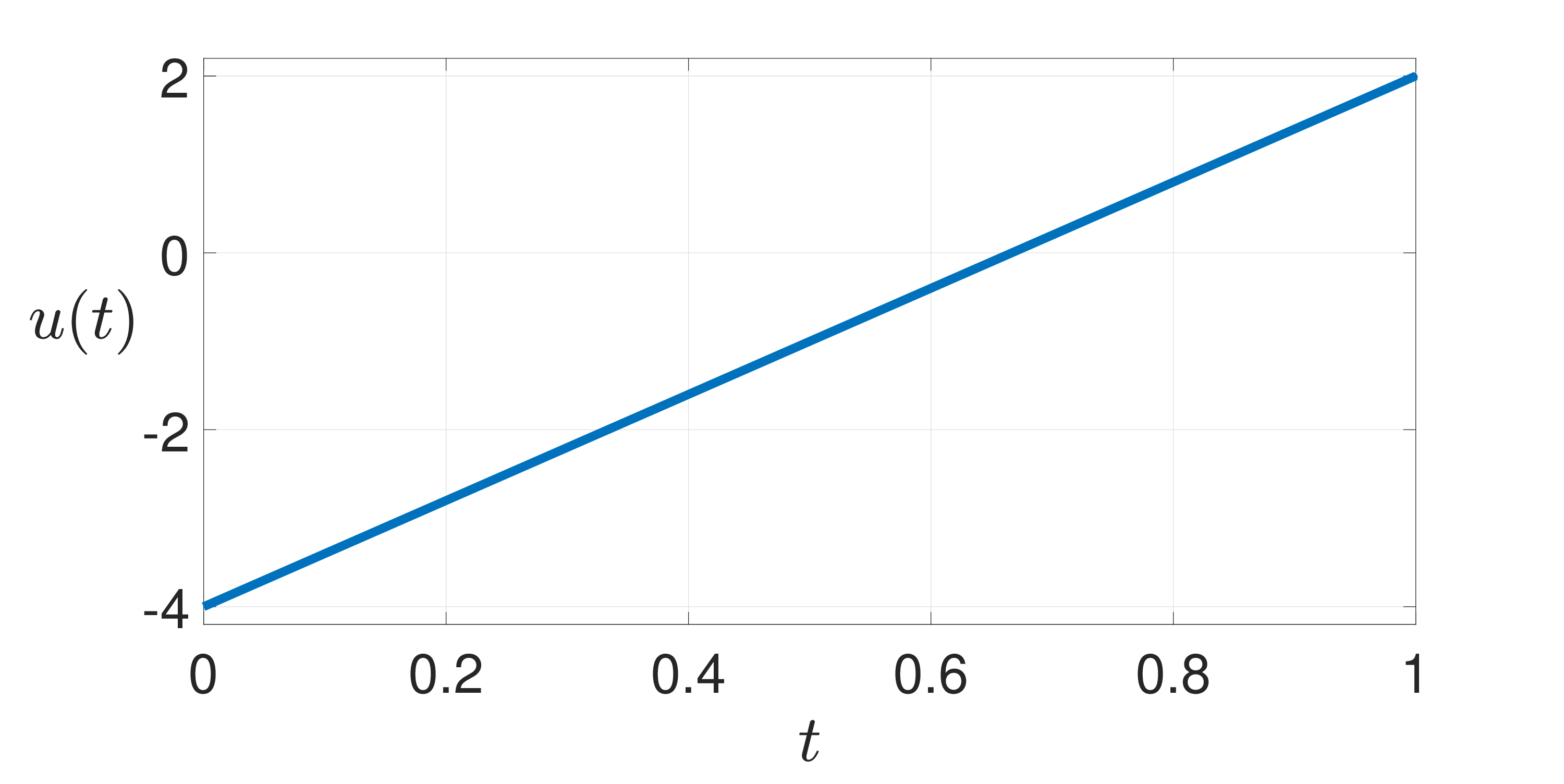}
  \caption{\small Optimal control variable.}
  \label{fig:xFigUnconstr}
\end{subfigure}
\hfill
\begin{subfigure}{.49\textwidth}
 \centering
  \includegraphics[width=\textwidth]{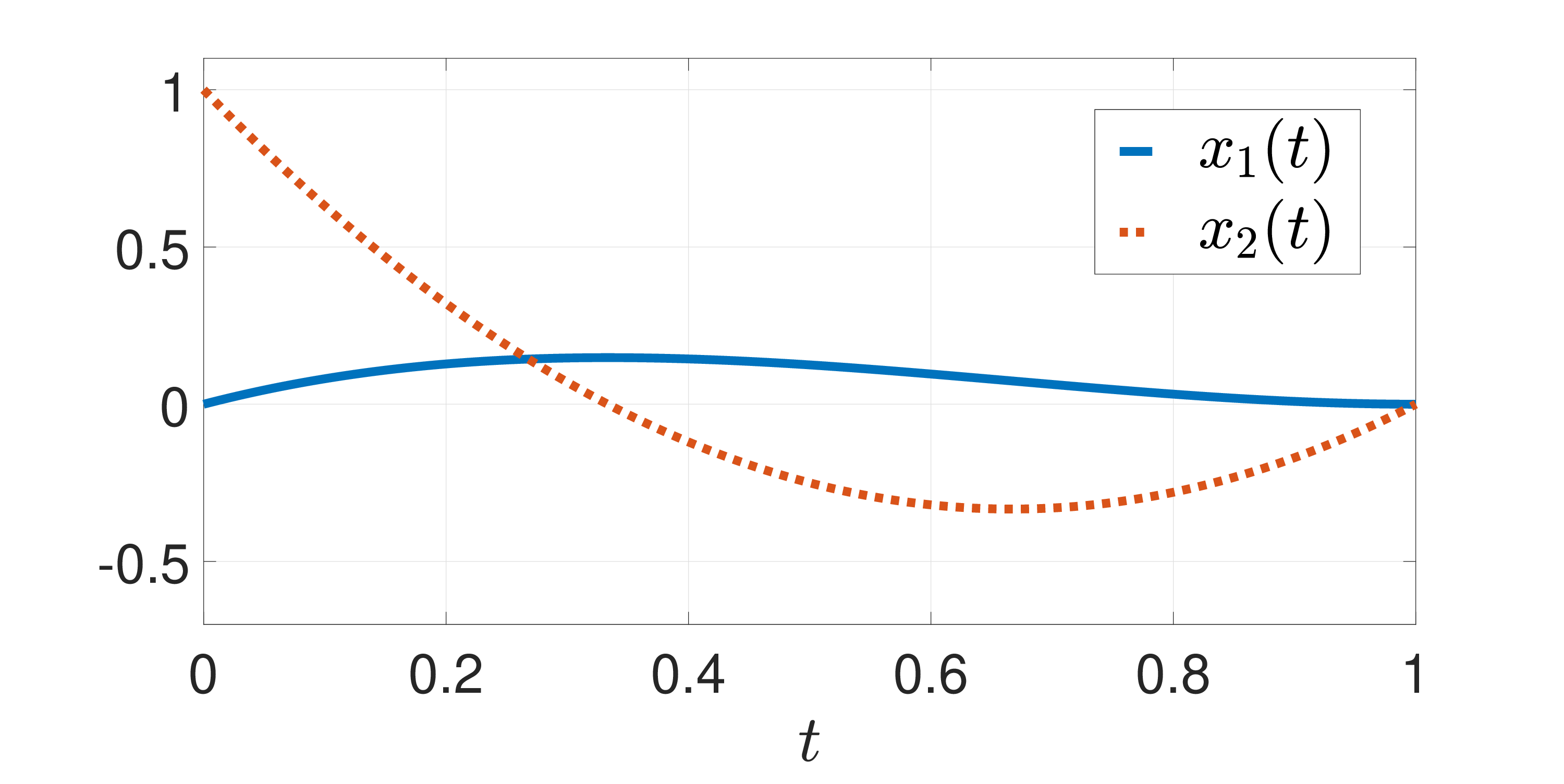}
  \caption{\small Optimal state variables.}
  \label{fig:uFigUnconstr}
  \end{subfigure}
  \caption{\sf Solution to Problem~(P) for $s_0$ = 0, $s_f$ = 0, $v_0$ = 1, and $v_f$ = 0.}
  \label{fig:solnUnconstr}
\end{figure}

\subsection{Generalized derivative}

For $a\in \R$, we define the {\em sign function} as follows.
\[
{\sign\, }(a):=\left\{
\begin{array}{rl}
1\,, & \hbox{ if }\ a>0\,,\\
-1\,, & \hbox{ if }\ a<0\,,\\
0\,, & \hbox{ if }\ a=0\,.
\end{array}
\right.
\]
Denote by $\beta:\R\to \R$ the {\em absolute value function}, i.e., $\beta(t)=|t|$ for every $t\in \R$. We now recall the {\em generalized derivative} of a function of the form $\theta(u):=|\eta(u)|$, where the function $\eta:{\cal L}^2([0,1];\R)\to \R$ is continuously differentiable. 
In Example 2.1.3 in \cite{Clarke}, the generalized derivative of $\beta(t)=|t|$ is given by
\begin{equation}\label{eq:abs}
    \partial \beta(t):=\left\{
    \begin{array}{cl}
        1\,, & \hbox{ if }\ t>0\,,  \\
        -1\,, & \hbox{ if }\ t<0\,,  \\
 \left[-1,1\right], & \hbox{ if }\ t=0\,.
    \end{array}
    \right.
\end{equation}
Since the function $\eta$ is continuously differentiable, then by \cite[Corollary in p.32]{Clarke}, the composition function $\theta$ satisfies the hypotheses of \cite[Theorem 2.3.10]{Clarke} which yields the following.
\[
\partial \theta(u) =\partial\beta(\eta(u))\,\dfrac{d \eta}{d u}(u)\,.
\]
Then using \eqref{eq:abs} we obtain the generalized derivative of $\theta$ as
\begin{equation}\label{eq:comp}
      \partial \theta(u):=\left\{
    \begin{array}{ll}
        \dfrac{d \eta}{d u}(u) & \hbox{ if }\ \eta(u)>0\,,  \\[4mm]
        -\dfrac{d \eta}{d u}(u) & \hbox{ if }\ \eta(u)<0\,,  \\[4mm]
 \left[-1,1\right]\,\dfrac{d \eta}{d u}(u) & \hbox{ if }\ \eta(u)=0\,.
    \end{array}
    \right. 
\end{equation}

\section{Main Results}
\label{sec:main}
We start this section by obtaining a closed expression for the control variables obtained by the PDP in Step 1(a). To preserve the flow of the paper, we have placed its proof in the Appendix. 

\begin{lemma}[Unique minimizer of the Lagrangian]
\label{lem:subgrad_p1}
Given $c>0$, re-write the Lagrangian in~\eqref{eq:LagrnCarUnconstr} in a concise form as
$L(\cdot,c):=\varphi(\cdot)+c\,|h_1(\cdot)|+c\,|h_2(\cdot)|$,
where  $h_1$ and $h_2$ are as in \eqref{eq:h_p1}, and 
$\varphi(u):=\ds\frac{1}{2}\int_0^1u^2(t)dt$.
Then the following holds.
\begin{itemize}
    \item[(i)] Define  the function $\gamma\in{\cal L}^2([0,1];\R)$ as $\gamma(t):=1-t$ for every $t\in [0,1]$. Given a fixed $u\in {\cal L}^2([0,1];\R)$, we have that
\begin{equation}\label{minimizerLagrnCarUnconstr}
u+c\,\alpha_1(u)\,\gamma+c\,\alpha_2(u) \in \partial_u L(u,c),
\end{equation}
where $\alpha_i(u)\in [-1,1]$, for $i=1,2$, are given by the following rule:
\begin{equation}\label{eq:alfa}
\left\{\begin{array}{lc}
  \alpha_i(u)={\rm sign\, }(h_i(u))   &  \hbox{ if }h_i(u)\neq 0, \\
 \alpha_i(u)\in [-1,1]    &  \hbox{ if }h_i(u)=0,
\end{array}\right.
\end{equation}
for $i=1,2$.
\item[(ii)]
Given $c>0$ there exists a unique minimizer $u_c$ of  $L(\cdot,c)$, and $u_c$ is given by
\begin{equation}\label{eq:minimizernCarUnconstr}
    u_c(t)=c\,\alpha_1(u_c) t-c(\alpha_1(u_c)+ \alpha_2(u_c)),
\end{equation}
where $\alpha_i$ are as in \eqref{eq:alfa}.
\end{itemize}
\end{lemma}
\begin{proof} See the proof of  Lemma~\ref{lem:subgrad_p1} in \hyperlink{prf:LEMsubgrad_p1}{Appendix}.
\end{proof}

\subsection{Partitioning the \boldmath{$r_1r_2$}-plane}

It will turn out that the analytical solution to Problem~(P) depends on the values of $r_1$ and $r_2$ uniquely defined in~\eqref{eq:r1_r2} in terms of the boundary conditions.  To study all possible cases based on the values of these parameters,
we consider the $r_1$- and $r_2$-axes and the lines $r_2=9\,r_1/5$, $r_2=3r_1$, which divide the $r_1r_2$--plane into 13 disjoint regions (Regions I--XIII) in the $r_1r_2$-plane. The descriptions of all 13 regions are given in Table~\ref{table:r1r2}.  Then we illustrate these regions in Figure~\ref{fig:r1r2relation}, with only eight of the (visually obvious) regions labelled. To complete the picture, Regions IX and X are the dashed lines in the first and third quadrants, respectively; Regions XI and XII are the positive and negative parts of the $r_1$-axis, respectively; Region XIII is the origin $(0,0)$.

  \begin{table}[t!]
    \centering
    {\small
    \begin{tabular}{cl|cl}
      Region & Set of $(r_1, r_2)$ & Region & Set of $(r_1, r_2)$\\[1mm]
       \hline &&& \\[-2mm]
        I & $r_1,r_2>0, \ r_2< 9r_1/5$ &
        V & $r_1,r_2<0, \ r_2>9r_1/5$ \\[2mm]
        II & $r_1,r_2>0, \ 9r_1/5<r_2\le 3r_1$ &
        VI & $r_1,r_2<0, \ 3r_1\le r_2<9r_1/5$ \\[2mm]
        III & $r_1,r_2>0, \ r_2>3r_1$ &
        VII & $r_1,r_2<0, \ r_2<3r_1$ \\[2mm]
        IV & $r_1\le 0, \ r_2>0$ &
        VIII & $r_1\ge 0,\, r_2<0$\\[2mm]
         \hline  &&& \\[-2mm]
        IX & $r_1>0,\, r_2=9r_1/5$&
        X & $r_1<0,\, r_2=9r_1/5$ \\[2mm]
        XI & $r_1> 0,\, r_2=0$&
        XII & $r_1< 0,\, r_2=0$ \\[2mm]
        XIII & $r_1= 0,\, r_2=0$\\[2mm]
        \hline
      \end{tabular}
      }
      \captionof{table}{\sf Regions I--XIII in the $r_1r_2$-plane.}
      \label{table:r1r2}
\end{table}
\begin{figure}[t!]
  \centering
    \includegraphics[width= 0.4\textwidth]{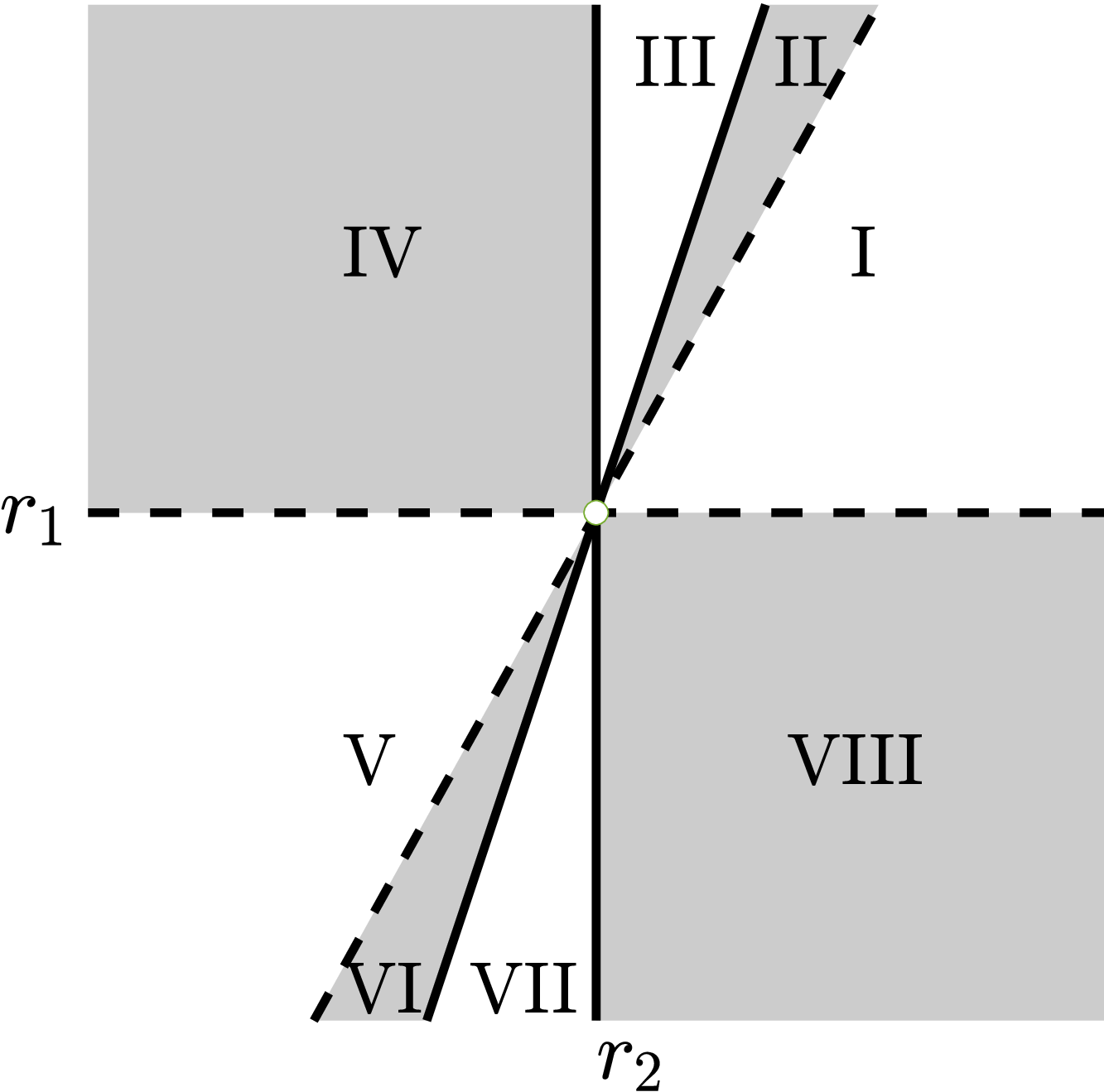}
    \captionof{figure}{\sf Regions I--VIII in the $r_1r_2$-plane. The solid boundary rays are contained by the grey regions II, IV, VI and VIII. Regions IX and X are the dashed lines in the first and third quadrants, respectively, and Regions XI and XII are the positive and negative parts of the $r_1$-axis, respectively. Region XIII is the origin $(0,0)$.}
    \label{fig:r1r2relation}
  \end{figure}
  
Define $M_1:=|12r_1-6r_2|$ and $M_2:=|-6r_1+4r_2|$. On each of the two lines $r_2=\dfrac{9}{5}r_1$ and $r_2=3r_1$, we have, respectively,
\begin{equation}\label{eq:maxIntersect}
 M_1 = M_2 =\bigg|\dfrac{6}{5}r_1\bigg| = \bigg|\dfrac{2}{3}r_2\bigg|
 \quad \hbox{and} \quad
 M_1 = M_2 =|6r_1| = |2r_2|.
\end{equation}
 We define
\[
    S_j:=\{(r_1,r_2)\in\R^2\,:\, (r_1,r_2)\ \mbox{is in  Region}\ j\}\,,\ \ \mbox{for}\,  j=I, II, III,\ldots, XIII.
\] 
It is trivial to check that the sets $S_j$ are mutually disjoint. Define $K:=\{(r_1,r_2,c)\in\R^3\,:\,c>0\}$.  Define also the following sets.
\begin{equation}
\begin{array}{l}\label{eq:r1r2cSet}
 V^1:=\left\{(r_1,r_2,c)\in K\,:\, (r_1,r_2)\in S_{I}\cup S_{II} \cup S_{III} \cup S_{IX},\,0<c<\min\left\{ \dfrac{6}{5}r_1,\dfrac{2}{3}r_2\right\}\right\}, \\[4mm]
 V^2:=\left\{(r_1,r_2,c)\in K\,:\, (r_1,r_2)\in S_{VII},\, -6r_1<c<-2r_2;\,\mbox{or}\,(r_1,r_2)\in S_{VIII},\,0<c<-2r_2\right\}, \\[3mm]
 V^3:=\left\{(r_1,r_2,c)\in K\,:\, (r_1,r_2)\in S_{III},\, 6r_1<c<2r_2;\, \mbox{or}\,(r_1,r_2)\in S_{IV},\,0<c<2r_2\right\}, \\[2mm]
 V^4:=\left\{(r_1,r_2,c)\in K\,:\, (r_1,r_2)\in S_{V}\cup S_{VI} \cup S_{VII} \cup S_{X}, \, 0<c<\min\left\{ -\dfrac{6}{5}r_1,-\dfrac{2}{3}r_2\right\}\right\}, \\[3mm]
 V^5:= \left\{(r_1,r_2,c)\in K\,:\, (r_1,r_2)\in S_{II},\, \dfrac{6}{5}r_1\le c<-6r_1+4r_2;\mbox{or}\,(r_1,r_2)\in S_{III},\,\dfrac{6}{5}r_1\le c\le6r_1\right\},\\[3mm]
V^6:=  \bigg\{(r_1,r_2,c)\in K\,:\, (r_1,r_2)\in S_{VI},\, -\dfrac{6}{5}r_1\le c<6r_1-4r_2; \mbox{or}\,(r_1,r_2)\in S_{VII},\,-\dfrac{6}{5}r_1\le c\le-6r_1\bigg\},\\[3mm]
V^7:=  \bigg\{(r_1,r_2,c)\in K\,:\, (r_1,r_2)\in S_{VII}\cup S_{VIII},\, -2r_2\le c<12r_1-6r_2; \\
\hspace{15mm} \mbox{or}\,(r_1,r_2)\in S_{I},\,\dfrac{2}{3}r_2\le c<12r_1-6r_2;\,\mbox{or}\,(r_1,r_2)\in S_{XI},\,0<c<12r_1-6r_2\bigg\},\\[3mm]
V^8:=  \bigg\{(r_1,r_2,c)\in K\,:\, (r_1,r_2)\in S_{III}\cup S_{IV},\, 2r_2\le c<-12r_1+6r_2; \\ 
\hspace{15mm} \mbox{or}\,(r_1,r_2)\in S_{V},\,-\dfrac{2}{3}r_2\le c<-12r_1+6r_2;\,\mbox{or}\,(r_1,r_2)\in S_{XII},\,0<c<-12r_1+6r_2\bigg\},\\[4mm]
V^9:=  \left\{(r_1,r_2,c)\in\R^2\times\R_{++}\,:\, c\ge\max\left\{|12r_1-6r_2|,|-6r_1+4r_2|\right\}\right\}.
\end{array}
\end{equation}

\begin{remark}[Union of \boldmath{$V^i$}]
\label{rk:cr_union}\rm
We can see from \eqref{eq:r1r2cSet} that $\bigcup\limits_{i=1}^{9} V^i\subseteq\R^2\times\R_{++}$.
\endproof
\end{remark}
In fact, the nine sets defined above are mutually disjoint. The proof of this fact is rather long, albeit elementary, so we have placed it in the Appendix. 
\begin{lemma}[Disjointness of \boldmath{$V^i$}]
\label{lm:crSet_disjoint}
The sets $V^i$s for $j=1,2,\ldots,9$ defined in Equation~\eqref{eq:r1r2cSet} are mutually disjoint, i.e., $V^i\cap V^j=\emptyset$ for all $i\not=j$, $i,j=1,\ldots,9$.
\end{lemma}
\begin{proof}
See the proof of  Lemma~\ref{lm:crSet_disjoint} in \hyperlink{prf:LEMcrSet_disjoint}{Appendix}.
\end{proof}

\begin{remark}[Nine possible cases for (P)] \rm 
Remark \ref{rk:cr_union} and Lemma \ref{lm:crSet_disjoint} imply that, for every instance $(r_1,r_2)$ of Problem (P), and every $c>0$, there exists a unique $i\in \{1,\ldots,9\}$ such that $(r_1,r_2,c)\in V^i$. Note also that since we have two constraints and three possible signs for each (counting sign zero when the constraint value is zero), there are nine possible sign cases for the constraints (see the first two columns of Table~\ref{table:thm_9cases}).  
It turns out that each of the nine inclusions $\{(r_1,r_2,c)\in V^i\}_{i=1,\ldots,9}$ corresponds to one and only one of the possible signs of the constraints.  Each of these nine cases, in turn, fully defines the formulas of the minimizer $u_c$ and the dual function (see the third and fourth columns of Table~\ref{table:thm_9cases}). We show these equivalences in Theorem~\ref{th:dualCarUnconstrIFF}. 
\proofbox
\end{remark}

The proof of the following theorem consists of the analysis of nine possible cases. Since each case follows a similar logical argument, we have placed the proof in the Appendix. 

\begin{theorem}[Dual function vs signs of constraint functions]
\label{th:dualCarUnconstrIFF}
Let $r_1$ and $r_2$ be the parameters associated with the boundary conditions in Problem (P), and let $V^i$s, $i = 1,2,\ldots,9$, be defined as in~\eqref{eq:r1r2cSet}. Given $c>0$, and $i = 1,2,\ldots,9$, call $u_c^{(i)}$ the (unique) minimizer of $L(\cdot,c)$ and call $q^{(i)}(c):=L(u_c,c)$ the dual function. Depending on the signs of $h_1(u_c)$ and $h_2(u_c)$, one and only one of the nine cases listed in Table~\ref{table:thm_9cases} holds.  For each case $i$, $i = 1,\ldots,9$, the statements in row $i$ of columns $(A)$, $(B)$ and $(C)$, referred to as $(iA)$, $(iB)$ and $(iC)$, are equivalent. 
\end{theorem}
\begin{proof}
See the proof of  Theorem~\ref{th:dualCarUnconstrIFF} in \hyperlink{prf:th:dualCarUnconstrIFF}{Appendix}.
\end{proof}

\begin{table}[t!]
\centering
{\small
\begin{tabular}{c| c| c| c}
Case & $(A)$ & $(B)$ & $(C)$ \\[1mm]
     & $(\sign(h_1(u_c)),\sign(h_2(u_c)))$ & Minimizer $u_c^{(i)}(t)$ & Dual function $q^{(i)}(c)$ \\[1mm]
     &  &  & with $(r_1,r_2,c)\in V^i$ \\[1mm] \hline \\[-3mm]
1 & $(+, +)$ & $ct-2c$ & $-\dfrac{7}{6}c^2+(r_1+r_2)c$ \\[3mm]
2 & $(+, -)$ & $ct$ & $-\dfrac{1}{6}c^2+(r_1-r_2)c$ \\[3mm]
3 & $(-, +)$ & $-ct$ & $-\dfrac{1}{6} c^2+(-r_1+r_2)c$ \\[3mm]
4 & $(-, -)$ & $ct+2c$ & $-\dfrac{7}{6}c^2-(r_1+r_2)c$ \\[3mm]
5 & $(0, +)$ & $\dfrac{-3c+6r_1}{2}t+\dfrac{c-6r_1}{2}$ & $-\dfrac{1}{8}c^2+(-\dfrac{3}{2}r_1+r_2)c+\dfrac{3}{2}r_1^2$ \\[3mm]
6 & $(0, -)$ & $\dfrac{3c+6r_1}{2}t-\dfrac{c+6r_1}{2}$ & $-\dfrac{1}{8}c^2+(\dfrac{3}{2}r_1-r_2)c+\dfrac{3}{2}r_1^2$ \\[3mm]
7 & $(+, 0)$ & $ct-\dfrac{c+2r_2}{2}$ & $-\dfrac{1}{24}c^2+(r_1-\dfrac{1}{2}r_2)c+\dfrac{1}{2}r_2^2$ \\[3mm]
8 & $(-, 0)$ & $-ct+\dfrac{c-2r_2}{2}$ & $-\dfrac{1}{24}c^2+(\dfrac{1}{2}r_2-r_1)c+\dfrac{1}{2}r_2^2$ \\[3mm]
9 & $(0, 0)$ & $(12r_1-6r_2)t-6r_1+2r_2$ & $6r_1^2-6r_1r_2+2r_2^2$ \\[1mm] \hline
\end{tabular}}
\caption{\sf Signs of $h_1(u_c)$ and $h_2(u_c)$, and the corresponding minimizer $u_c^{(i)}(t)$ of $L(\cdot,c)$ and dual function $q^{(i)}(c)$ with $(r_1,r_2,c)\in V^i$, $i = 1,\ldots,9$.}
\label{table:thm_9cases}
\end{table}

\begin{remark}[Regions I--XIII and associated dual functions] \rm
In Table~\ref{table:r1r2Dual}, we list the dual functions associated with each of the 13 regions defined in Table~\ref{table:r1r2}. The first column of Table~\ref{table:r1r2Dual} lists the regions to which $(r_1,r_2)$ belongs, the second column describes the corresponding dual function, and the last column provides examples with certain boundary conditions. For example, in the first row of Table~\ref{table:r1r2Dual}, we have the boundary condition vector $(s_0,s_f,v_0,v_f)=(0,0,1,0)$. In this case, $(r_1,r_2) = (1,1)$ belongs to Region I, where the dual function $q(\cdot)$ for Problem~(P) is a concatenation of the three quadratic functions $q^{(1)}(\cdot)$, $q^{(7)}(\cdot)$ and $q^{(9)}(\cdot)$:  $q(c) = q^{(1)}(c)$ if $c\in(0,2r_2/3)$; $q(c) = q^{(7)}(c)$ if $c\in[2r_2/3,12r_1-6r_2)$  and $q(c) = q^{(9)}(c)$ if $c\in[12r_1-6r_2,+\infty)$. The following corollary illustrates how Theorem \ref{th:dualCarUnconstrIFF} is applied to determine the structure of the dual function for these initial conditions. All other rows in Table~\ref{table:r1r2Dual} are obtained and interpreted similarly. 
\proofbox
\end{remark}

\begin{corollary}[Solution for (P) with example boundary values]
\label{cor:boundary01_p1}
When Problem~(P) has boundary values $v_0=1$, $v_f=0$, $s_0=0$, and $s_f=0$, the solution for the control variable and the dual function for Problem (P) are given by
\begin{equation*}
\begin{split}
u_c(t) = \left\{
\begin{array}{ll}
ct-2c\,, \quad &  c\in\left(0,\dfrac{2}{3}\right), \\[4mm]
ct-\dfrac{c+2}{2}\,,  &  c\in\left[\dfrac{2}{3},6\right),\\[4mm]
6t-4\,,  &  c\in [6,+\infty),
\end{array}
\right.
\end{split}\quad
\begin{split}
q(c) = \left\{
\begin{array}{ll}
-\dfrac{7}{6}c^2+2c\,, \quad &  c\in\left(0,\dfrac{2}{3}\right), \\[4mm]
-\dfrac{c^2}{24}+\dfrac{c}{2}+\dfrac{1}{2}\,,  &  c\in\left[\dfrac{2}{3},6\right),\\[4mm]
2,  &  c\in [6,+\infty)\,.
\end{array}
\right.
\end{split}
\end{equation*}
Note that in this situation, $u_c(t) = 6\,t - 4$ solves (P) for $c\ge 6$.
\end{corollary}
\begin{proof}
When the values of $v_0$, $v_f$, $s_0$, $s_f$ are substituted one gets $r_1=v_0-s_f+s_0=1$ and $r_2=v_0-v_f=1$.
Hence, $(r_1,r_2)$ is in Region I. It is easy to check, from the definitions of the $V^i$s, that for $c\in (0,2/3)$ we have $(r_1,r_2,c)\in V^1$. This corresponds to Case 1 of Theorem~\ref{th:dualCarUnconstrIFF}. Therefore, we must have $q(c)=q^{(1)}(c)=-\dfrac{7}{6}
c^2 + (r_1 + r_2)c=-\dfrac{7}{6}
c^2 + 2c$ (because $r_1=r_2=1$). By Theorem~\ref{th:dualCarUnconstrIFF}, this dual function, together with the fact that $(r_1,r_2,c)\in V^1$, is equivalent to having case $(1A)$, namely that $h_1(u_c)>0$ and $h_2(u_c)>0$.
This implies that $\alpha_1(u_c)=1$ and $\alpha_2(u_c)=1$. Hence, \eqref{eq:minimizernCarUnconstr}  gives  $u_c(t)=ct-2c$.\\
For $c\in \left[\dfrac{2}{3},6\right)$, it is easy to check that $(r_1,r_2,c)\in V^7$, and hence Case 7 of Theorem~\ref{th:dualCarUnconstrIFF} holds. According to this theorem and substituting $r_1=r_2=1$ in the expression for $q$ in $(7C)$, we must have the dual function $q(c)=q^{(7)}(c)=-\dfrac{c^2}{24}+\dfrac{c}{2}+\dfrac{1}{2}$. Since $(7C)$ implies $(7A)$, we must have $h_1(u_c)>0$ and $h_2(u_c)=0$. Therefore, $\alpha_1(u_c)=1$. By \eqref{eq:minimizernCarUnconstr} we have $u_c(t)=ct-c-c\alpha_2(u_c)$. Using $c\alpha_2(u_c)=\dfrac{-c+2r_2}{2}$ from \eqref{eq:cr_case7_1} and $r_2=1$, we obtain $u_c(t)=ct-\dfrac{c+2}{2}$.\\
Finally, note that $(r_1,r_2,c)\in V^9$ whenever $c\ge 6$. Therefore, we are in Case 9 of Theorem~\ref{th:dualCarUnconstrIFF} and hence $h_1(u_c)=h_2(u_c)=0$. In this case, the dual function becomes constant, and substituting $r_1=r_2=1$ in the expression for $q$ in $(9C)$ gives
 $q(c)=q^{(9)}=2$ for $c\in [6,+\infty)$. Using \eqref{eq:uabUnconstr} and \eqref{eq:a&b_case9} in the proof of Theorem~\ref{th:dualCarUnconstrIFF}, we obtain the control variable $u_c(t)=6t-4$. Since $u_c$ is feasible and $q$ is a concave function that becomes constant for $c\ge 6$, $q$ has been maximized and therefore $u_c(t)=6t-4$ is the solution of (P).
\end{proof}

\begin{remark}[Case 9, \boldmath{$(r_1,r_2)=(0,0)$}, is equivalent to \boldmath{$u_c=0$} for all \boldmath{$c>0$}]
\label{rem:r1_r2_0} \rm
In\linebreak Theorem~\ref{th:dualCarUnconstrIFF}, Case 9 corresponds to $(r_1,r_2)=(0,0)$. Recall from Equation~\eqref{eq:r1_r2} that 
\[
r_1=v_0+s_0-s_f, \quad \mbox{and} \quad r_2=v_0-v_f,
\]
It is easy to check that the case $(r_1,r_2)=(0,0)$ corresponds to Case 9 in Theorem~\ref{th:dualCarUnconstrIFF}. Using column (B) in Table~\ref{table:thm_9cases} for $(r_1,r_2)=(0,0)$ 
gives $u_c(\cdot)=0$ for all $c>0$. This means no external force is imposed on the car. Physically, the car keeps the initial state from the beginning. Namely, if the car is moving with a certain speed in the beginning, it will keep that speed and keep moving; if the car is initially stationary, then it will remain so. Indeed, this can be interpreted from further computation of the final state. Using $r_1=0$ and the first equation in \eqref{eq:r1_r2}, we obtain $s_f-s_0=v_0\cdot 1$, which can be interpreted as the displacement is moving at the initial velocity for the whole time. Using $r_2=0$ and the second equation in \eqref{eq:r1_r2}, we obtain $v_0=v_f$, which can be interpreted as the initial velocity equals to the final velocity. 

Conversely, when $u_c(\cdot)=0$ for all $c>0$, we must have $(r_1,r_2)=(0,0)$. Indeed, let us refer to Table~\ref{table:thm_9cases}, and check the nine possible expressions for $u_c(\cdot)$. Since $c>0$, we have trivially that $u_c(\cdot)\neq0$ for cases 1--4 and cases 7, 8. We now show that $u_c(\cdot)\neq0$ for Case 5. In this case, $u_c(t)=\dfrac{-3c+6r_1}{2}t+\dfrac{c-6r_1}{2}$. Setting both coefficients of this affine function to $0$, we have $c=2r_1$ and $c=6r_1$. Since this equality holds for all $c>0$,  we have a contradiction. A similar argument shows that $u_c(\cdot)\not=0$ 
for Case 6. Finally, let us check whether $u_c(\cdot)$ can be $0$ in Case 9. In the latter case, we have $u_c(t)=(12r_1-6r_2)t-6r_1+2r_2$. This expression gives us the following simultaneous equations
\[
 \left\{
\begin{array}{rl}
12r_1-6r_2=0,\\
-6r_1+2r_2=0,
\end{array}
\right.
\]
which yield $r_1=0$ and $r_2=0$. Altogether, we have shown that the control variable $u_c(\cdot)=0$ for all $c>0$ if and only if $(r_1,r_2)=(0,0)$.
\endproof
\end{remark}

\begin{table}[!h]
\centering
{\small
\begin{tabular}{l | l| c}
 Region &  Dual function, $q(c)$ & Example B.C. \\
 \hline\hline &&\\[-4mm]
 $\begin{array}{l}
 \mbox{I} \\
 (r_1,r_2>0, \\
 \ r_2<9r_1/5)
 \end{array}$
 & 
 $ 
 \begin{array}{ll}
 q^{(1)}(c), &\mbox{ if } 0<c < 2r_2/3\,, \\
 q^{(7)}(c), &\mbox{ if } 2r_2/3\le c<12r_1-6r_2\,, \\
 q^{(9)}(c), &\mbox{ if } c\ge12r_1-6r_2\,.
 \end{array}$
& 
$\begin{array}{l}
s_0=0,\ s_f=0,\\
v_0=1,\ v_f=0
\end{array}$
\\[6mm] \hline &&\\[-4mm]
$\begin{array}{l}
 \mbox{II} \\
 (r_1,r_2>0, \\
 \ 9r_1/5<r_2\le 3r_1)
 \end{array}$
 &
  $ 
\begin{array}{ll}
 q^{(1)}(c), &\mbox{ if } 0<c<6r_1/5\,, \\
 q^{(5)}(c), &\mbox{ if } 6r_1/5\le c<-6r_1+4r_2\,, \\
 q^{(9)}(c), &\mbox{ if } c\ge-6r_1+4r_2\,.
 \end{array}$
&
$\begin{array}{l}
s_0=0,\ s_f=1,\\
v_0=2,\ v_f=-1/2.
\end{array}$
\\[6mm] \hline &&\\[-4mm]
$\begin{array}{l}
 \mbox{III} \\
 (r_1,r_2>0, \\
 \ r_2>3r_1)
 \end{array}$
 &
  $ 
\begin{array}{ll}
 q^{(1)}(c), &\mbox{ if } 0<c<6r_1/5\,, \\
 q^{(5)}(c), &\mbox{ if } 6r_1/5\le c\le 6r_1\,, \\
 q^{(3)}(c), &\mbox{ if } 6r_1<c<2r_2\,,\\
 q^{(8)}(c), &\mbox{ if } 2r_2\le c<-12r_1+6r_2\,, \\
 q^{(9)}(c), &\mbox{ if } c\ge-12r_1+6r_2\,.\\
 \end{array}$
&
$\begin{array}{l}
s_0=0,\ s_f=3,\\
v_0=4,\ v_f=0.
\end{array}$
\\[11mm] \hline &&\\[-4mm]
$\begin{array}{l}
 \mbox{IV} \\
 (r_1\le 0, \ r_2>0)
 \end{array}$
 &
  $
 \begin{array}{ll}
 q^{(3)}(c), &\mbox{ if } 0<c<2r_2\,,\\
 q^{(8)}(c), &\mbox{ if } 2r_2\le c<-12r_1+6r_2\,, \\
 q^{(9)}(c), &\mbox{ if } c\ge-12r_1+6r_2\,.\\
 \end{array}$
&
$\begin{array}{l}
s_0=0,\ s_f=2,\\
v_0=1,\ v_f=0.
\end{array}$
\\[6mm] \hline &&\\[-4mm]
$\begin{array}{l}
 \mbox{V} \\
 (r_1, r_2<0, \\
 \ r_2>9r_1/5)
 \end{array}$
 &
  $
\begin{array}{ll}
 q^{(4)}(c), &\mbox{ if } 0<c<-2r_2/3\,,\\
 q^{(8)}(c), &\mbox{ if } -2r_2/3\le c<-12r_1+6r_2\,, \\
 q^{(9)}(c), &\mbox{ if } c\ge-12r_1+6r_2\,.\\
 \end{array}$
&
$\begin{array}{l}
s_0=0,\ s_f=1,\\
v_0=0,\ v_f=1.
\end{array}$
\\[6mm] \hline &&\\[-4mm]
$\begin{array}{l}
 \mbox{VI} \\
 (r_1, r_2<0, \\
 \ 3r_1\le r_2<9r_1/5)
 \end{array}$
 &
  $
\begin{array}{ll}
 q^{(4)}(c), &\mbox{ if } 0<c<-6r_1/5\,,\\
 q^{(6)}(c), &\mbox{ if } -6r_1/5\le c<6r_1-4r_2\,, \\
 q^{(9)}(c), &\mbox{ if } c\ge6r_1-4r_2\,.\\
 \end{array}$
&
$\begin{array}{l}
s_0=0,\ s_f=5,\\
v_0=0,\ v_f=12.
\end{array}$
\\[6mm] \hline &&\\[-4mm]
$\begin{array}{l}
 \mbox{VII} \\
 (r_1, r_2<0, \\
 \ r_2<3r_1)
 \end{array}$
 &
  $
\begin{array}{ll}
 q^{(4)}(c), &\mbox{ if } 0<c<-6r_1/5\,,\\
 q^{(6)}(c), &\mbox{ if } -6r_1/5\le c\le-6r_1\,, \\
 q^{(2)}(c), &\mbox{ if } -6r_1<c<-2r_2\,, \\
 q^{(7)}(c), &\mbox{ if } -2r_2\le c<12r_1-6r_2\,, \\
 q^{(9)}(c), &\mbox{ if } c\ge12r_1-6r_2\,.\\
 \end{array} $
&
$\begin{array}{l}
s_0=0,\ s_f=1,\\
v_0=0,\ v_f=5.
\end{array}$
\\[11mm] \hline &&\\[-4mm]
$\begin{array}{l}
 \mbox{VIII} \\
 (r_1\ge 0, \ r_2<0)
 \end{array}$
 &
  $
 \begin{array}{ll}
 q^{(2)}(c), &\mbox{ if } 0<c<-2r_2\,, \\
 q^{(7)}(c), &\mbox{ if } -2r_2\le c<12r_1-6r_2\,, \\
 q^{(9)}(c), &\mbox{ if } c\ge12r_1-6r_2\,.\\
 \end{array} $
&
$\begin{array}{l}
s_0=1,\ s_f=0,\\
v_0=0,\ v_f=1.
\end{array}$ 
\\[6mm] \hline &&\\[-4mm]
  $\begin{array}{ll}
 \mbox{IX} \, (r_1>0, \\
 r_2= 9r_1/5)
 \end{array}$
 &
 $ 
\begin{array}{ll}
 q^{(1)}(c), &\mbox{ if } 0<c<6r_1/5\,, \\
 q^{(9)}(c), &\mbox{ if } c\ge 6r_1/5(=12r_1-6r_2)\,.
 \end{array} $
 &
 $\begin{array}{l}
s_0=0,\ s_f=0,\\
v_0=5,\ v_f=-4.
\end{array}$
\\[4mm] \hline &&\\[-3mm]
   $\begin{array}{ll}
 \mbox{X} \, (r_1<0, \\
 r_2= 9r_1/5)
 \end{array}$
 &
$
 \begin{array}{ll}
 q^{(4)}(c), &\mbox{ if } 0<c<-6r_1/5\,, \\
 q^{(9)}(c), &\mbox{ if } c\ge -6r_1/5(=-12r_1+6r_2)\,.
 \end{array} $
 &
 $\begin{array}{l}
s_0=0,\ s_f=1,\\
v_0=0,\ v_f=9/5.
\end{array}$
\\[4mm] \hline &&\\[-3mm]
  $\begin{array}{ll}
  \mbox{XI}  \\
 (r_1>0,\, r_2=0)
 \end{array}$
 &
 $
 \begin{array}{ll}
 q^{(7)}(c), &\mbox{ if } 0<c<12r_1-6r_2\,, \\
 q^{(9)}(c), &\mbox{ if } c\ge 12r_1-6r_2\,.
 \end{array} $
 &
 $\begin{array}{l}
s_0=1,\ s_f=0,\\
v_0=0,\ v_f=0.
\end{array}$
 \\[4mm] \hline &&\\[-3mm]
  $\begin{array}{ll}
   \mbox{XII}\\
 (r_1<0,\,r_2=0)
 \end{array}$
 &
 $
 \begin{array}{ll}
 q^{(8)}(c), &\mbox{ if } 0<c<-12r_1+6r_2\,, \\
 q^{(9)}(c), &\mbox{ if } c\ge -12r_1+6r_2\,.
 \end{array} $
 &
  $\begin{array}{l}
s_0=0,\ s_f=1,\\
v_0=0,\ v_f=0.
\end{array}$
 \\[4mm] \hline 
   $\begin{array}{ll}
   \mbox{XIII}\\
 (r_1=0,\,r_2=0)
 \end{array}$
 &
  $
  \begin{array}{ll}
 q^{(9)}(c), &\mbox{ if } c> 0\,.
 \end{array}$
 &
  $\begin{array}{l}
s_0=0,\ s_f=0,\\
v_0=0,\ v_f=0.
\end{array}$\\
\hline
 \end{tabular}
 }
\caption{\sf The dual function in each of the regions as listed in Table~\ref{table:r1r2} and displayed in Figure~\ref{fig:r1r2relation}. The notation $q^{(i)}$, $i = 1,\ldots,9$, indicates the dual function derived in Case $i$ of Theorem~\ref{th:dualCarUnconstrIFF}.  The dual function $q(c)$ for a region is a concatenation of $q^{(i)}(c)$ as described.}
\label{table:r1r2Dual}
\end{table}

\clearpage

In what follows, we denote by $M_P$ the optimal value of Problem~(P), and by $S(P)$ and $D(P)$ the sets of primal and dual solutions of Problem~(P), respectively.

Next, we define the exact and least exact penalty parameters.  A positive number $\bar c$ is an {\em exact penalty parameter} for Problem~(P) if for every $c\ge \bar c$ we have ${\rm argmin \,}L(\cdot, c)\subset S(P)$.  In particular, $q(c)=M_P$, for $c\ge \bar c$.  On the other hand, a positive number $\bar c$ is the {\em least exact penalty parameter} for Problem~(P) when the following two properties hold.
\begin{itemize}
    \item[(i)] $\bar c$ is an {exact penalty parameter} for Problem~(P), and 
    \item[(ii)] for every $c< \bar c$ we have ${\rm argmin}\, L(\cdot,c)\cap S(P)=\emptyset$.
\end{itemize}

\begin{corollary}[Analytical solution for the least exact penalty parameter]
\label{cor:solution_exactParameter_p1}
With the assumptions of Theorem \ref{th:dualCarUnconstrIFF}, the following hold.
\begin{itemize}
    \item [$i)$] The control and state variables obtained by the PDP method as solution of Problem~(P) comply with the analytical solutions \eqref{solnCarUnconstr} and \eqref{cPenaltyCarUnconstr}.
    \item[$ii)$]The least exact penalty parameter for Problem~(P) is given by
    \[
    \bar c = \max\{|12r_1-6r_2|,|-6r_1 + 4r_2|\}\,, 
    \]
    with the optimal dual value $\bar{q} := q(\bar{c}) = 6r_1^2-6r_1r_2+ 2r_2^2$.
\end{itemize}
\end{corollary}
\begin{proof}
\ $ (i)$ By the stopping criteria in iteration $k$ of the PDP method, we have $z_k=0$ (Algorithm~\ref{alg:PDP}, Step 1(b)), which means that $\sigma(h(u_{c_k}))=|h_1(u_{c_k})|+|h_2(u_{c_k})|=0$. This happens if and only if $h_1(u_{c_k})=0$ and $h_2(u_{c_k})=0$, which corresponds to Case 9 in Theorem~\ref{th:dualCarUnconstrIFF}. Let us now compute the optimal control $u_{c_k}(t)$ for this case. From the proof of Theorem \ref{th:dualCarUnconstrIFF}, Case 9, equation \eqref{eq:a&b_case9}, we have $a=12r_1-6r_2$, and $b=-6r_1+2r_2$. Substituting $r_1$ and $r_2$ given in \eqref{eq:r1_r2} into $a$ and $b$, we obtain
\begin{equation}\label{eq:abP1}
 a=-12(s_f-s_0)+6(v_0+v_f)\quad \mbox{and} \quad b=6(s_f-s_0)-2(2v_0+v_f).   
\end{equation}
Combining the assumption that $u_{c_k}(t) =at+b$ from \eqref{eq:uabUnconstr}, we have established that the optimal control complies with \eqref{solnCarUnconstr} and \eqref{cPenaltyCarUnconstr}. We now show the solution for the optimal state variables $x_1(t)$ and $x_2(t)$. Recall the expression~\eqref{eq:xP1reformulation} for $x_1(t)$ and $x_2(t)$  from the reformulated Problem~$(P1)$. We substitute $u_{c_k}$ of \eqref{eq:uabUnconstr} into \eqref{eq:xP1reformulation} and obtain
\begin{equation}\label{eq:abxP1}
    x_2(t)=\dfrac{1}{2}at^2+bt+v_0 \quad \mbox{and} \quad
    x_1(t)=\dfrac{1}{6}at^3+\dfrac{1}{2}bt^2+v_0t+s_0,
\end{equation}
together with \eqref{eq:abP1}, we have established the optimal state variables comply with the solutions depicted in  \eqref{solnCarUnconstr} and \eqref{cPenaltyCarUnconstr}.

We now show $(ii)$.  By Theorem~\ref{th:dualCarUnconstrIFF} and the definition of $\bar c$, if $c\ge \bar c$ then we are in Case 9 and hence $(h_1(u_c),h_2(u_c))=(0,0)$. Hence, the (unique) Lagrangian minimizer $\{u_c\}={\rm argmin \,}L(\cdot, c)$ is a solution for $c\ge \bar c$. Indeed, in this situation, we can write
\[
\bar q\ge q(c)=L(u_c,c)=\varphi(u_c)\ge M_P,
\]
where we used in the first equality the definition of $q$ and the fact that $u_c$ is a minimizer of $L(\cdot,c)$. In the second equality and in the last inequality we used the fact that $h_1(u_c)=h_2(u_c)=0$ so $u_c$ is feasible for (P). By strong duality, all values in the expression above coincide, and we deduce that whenever $c\ge \bar{c}$ the solution $u_c\in S(P)$ and $c\in S(D)$. Consequently, $q(c)= \bar q=M_P$ for $c\ge \bar c$. Therefore, $\bar c$ is an exact penalty parameter. To verify that  $\bar c$ is the least exact penalty parameter, we need to show that $\bar c$ satisfies condition (ii) of the definition of the least exact penalty parameter. That is, we will show that ${\rm argmin \,} L(\cdot,c)\cap S(P)=\emptyset$ if $c<\bar c$. Since $c<\bar c$, by Theorem~\ref{th:dualCarUnconstrIFF}, Case 9 cannot happen and only Cases 1--8 can happen for this value of $c$. For all these eight cases, we have $|h_1(u_c)|+|h_2(u_c)|>0$. In particular, $u_c$ is not feasible for Problem~(P), and hence $u_c\not\in S(P)$. Consequently, ${\rm argmin \,}L(\cdot,c)\cap S(P)= \{u_c\}\cap S(P)=\emptyset$ and condition (ii) of the definition of the least exact penalty parameter holds.
\end{proof}

\subsection{Behaviour of the constraint functions}
\label{sec:constraints_p1}
The function $h:=(h_1,h_2)$ defined in \eqref{eq:h_p1} describes the constraints of Problem~(P). For a fixed $c>0$, the uniqueness of the minimizer $u_c$ of the Lagrangian allows us to define a map $c\mapsto (h_1(u_c),h_2(u_c))$. More precisely, define $z:\R_+ \to \R^2$ such that
 \begin{equation}\label{eq:z_P1}
 z(0):= \left[ 
 \begin{array}{l}
     r_1  \\[2mm]
     r_2 \\
 \end{array}
 \right]; \quad \mbox{and} \quad
 z(c):=  \left[ 
 \begin{array}{l}
     z_1(c)  \\[2mm]
     z_2(c) \\
 \end{array}
 \right]
 :=\left[ 
 \begin{array}{l}
     h_1(u_c)  \\[2mm]
     h_2(u_c) \\
 \end{array}
 \right],\ \mbox{ for} \: c>0.
 \end{equation}

By means of the proof of Theorem \ref{th:dualCarUnconstrIFF}, and using \eqref{eq:h_alpha}, we obtain an explicit expression for $(h_1(u_c),h_2(u_c))$ as a function of $c$, for each of the cases, according to the signs of $h_1(u_c)$ and $h_2(u_c)$.  

The definitions in~\eqref{eq:z_P1} provide parametric equations with parameter $c$. The function $z(c)$  provides a measure of feasibility for the corresponding $u_c$ at each value of $c$ as updated by the PDP algorithm. This allows us to study the behaviour of $z$ as the dependent variable of a dynamical system with the independent variable $c$. Then, the trajectory $z(c)$ will effectively show the behaviour, or the evolution of the constraint functions $h_1(u_c)$ and $h_2(u_c)$ towards $(0,0)$ (the solution) as $c$ increases. 

We illustrate various trajectories $z(c)$ in Figure~\ref{fig:subfiguresz1z2}, starting at various points $(r_1,r_2)$, depending on the conditions described in Table~\ref{table:r1r2Dual}. Recall that there are 13 different zones, or regions, where $(r_1,r_2)$ might be. With the application of our PDP algorithm the trajectory $z(c)$ always terminates at $z(\overline{c})=(0,0)$ (i.e., we have a feasible point), the origin of Figure~\ref{fig:subfiguresz1z2}, where the constraints are satisfied with $(h_1(u_{\bar c}),h_2(u_{\bar c}))=(0,0)$, and $\overline{c}$ is the least exact penalty parameter (or greater if we overstep).

\begin{figure}[t!]
  \begin{subfigure}{\textwidth}
  \centering
  \includegraphics[width=0.42\textwidth]{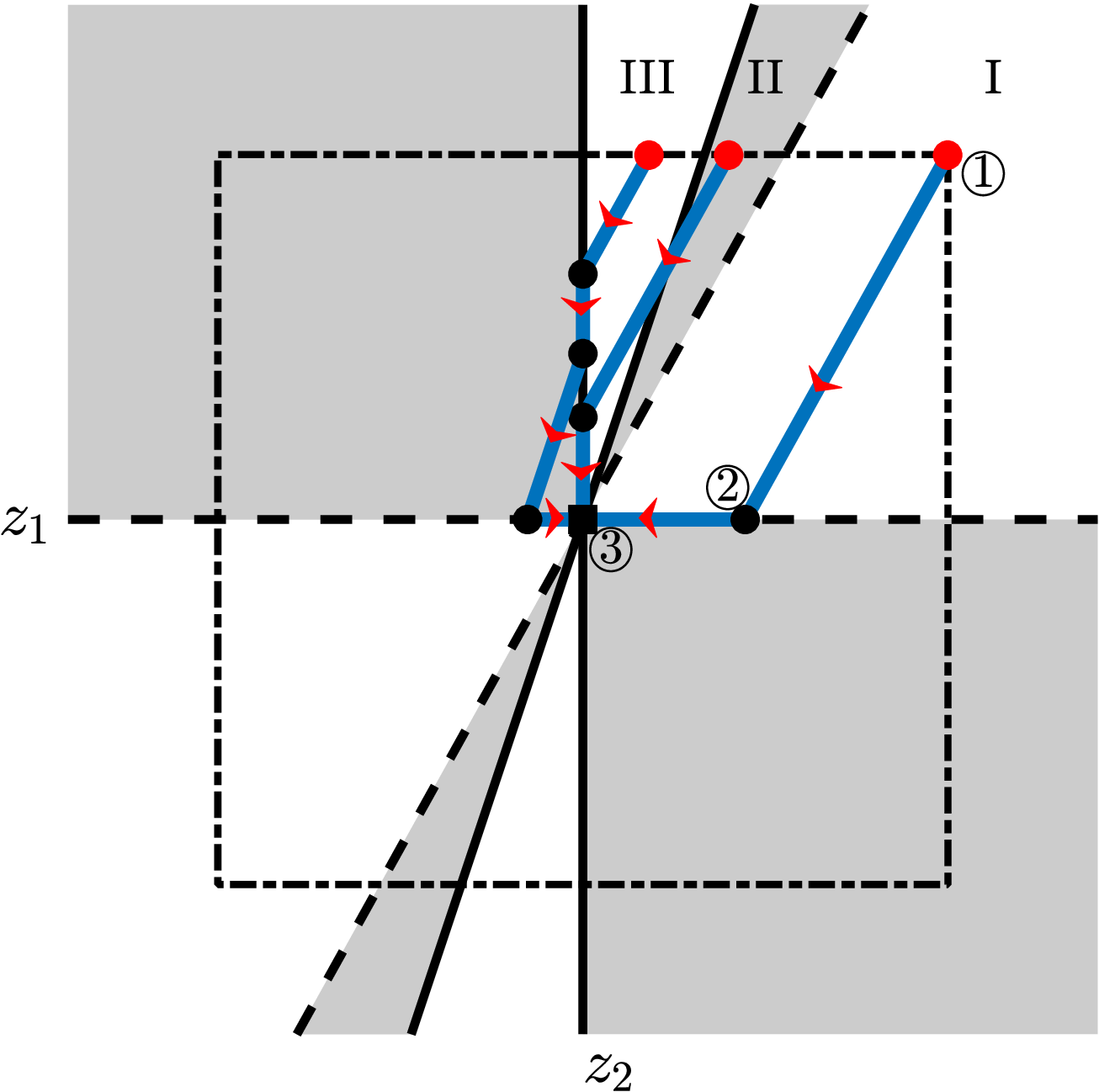} \hspace{7mm}  
  \includegraphics[width=0.42\textwidth]{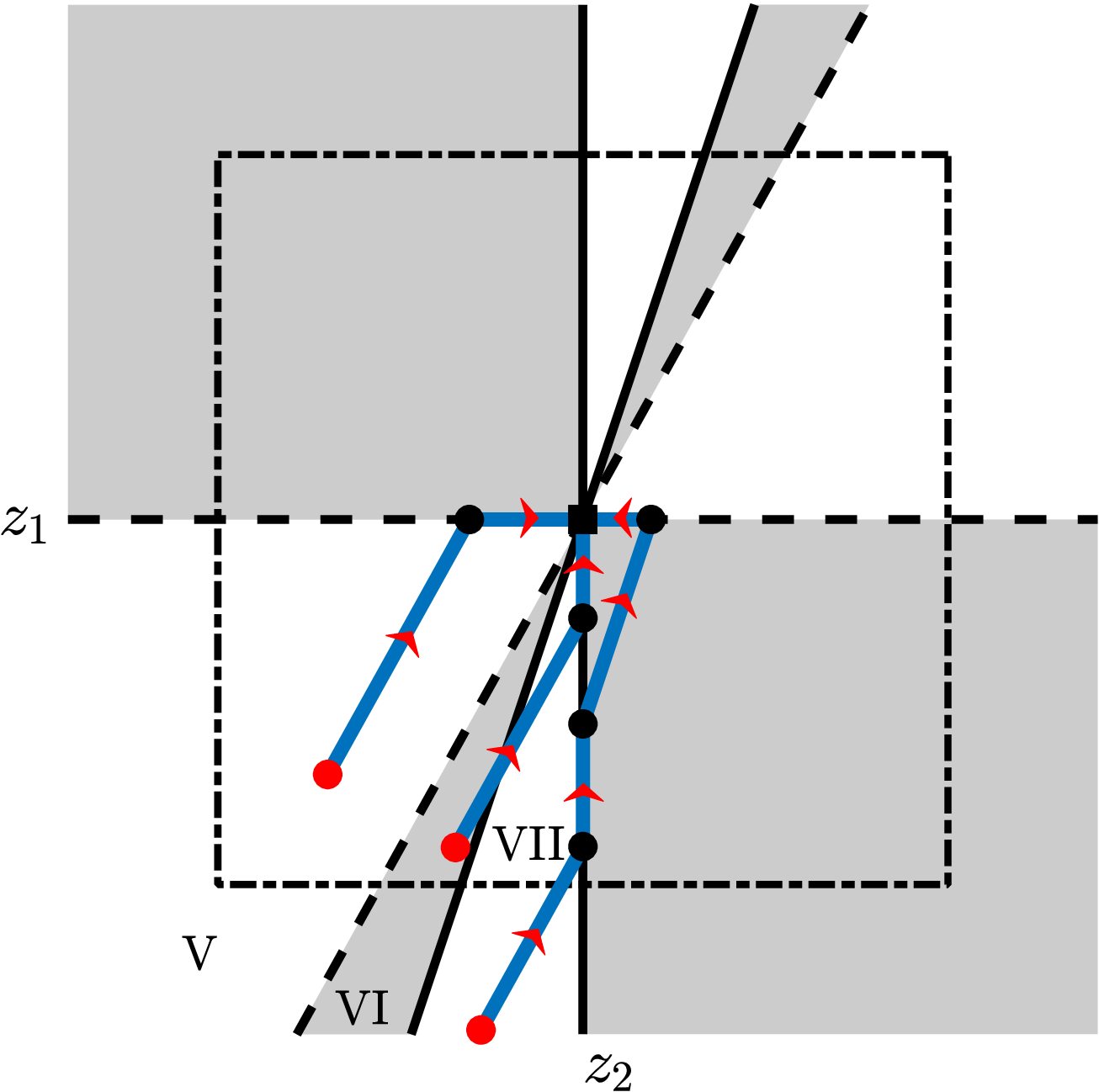}
  \caption{\small\sf Starting in Regions I, II, III (left) and V, VI, VII (right).}
  \label{fig:h1h2_1_3}
\end{subfigure}
\ \\[2mm] 
\begin{subfigure}{\textwidth}
\centering
\includegraphics[width=0.42\textwidth]{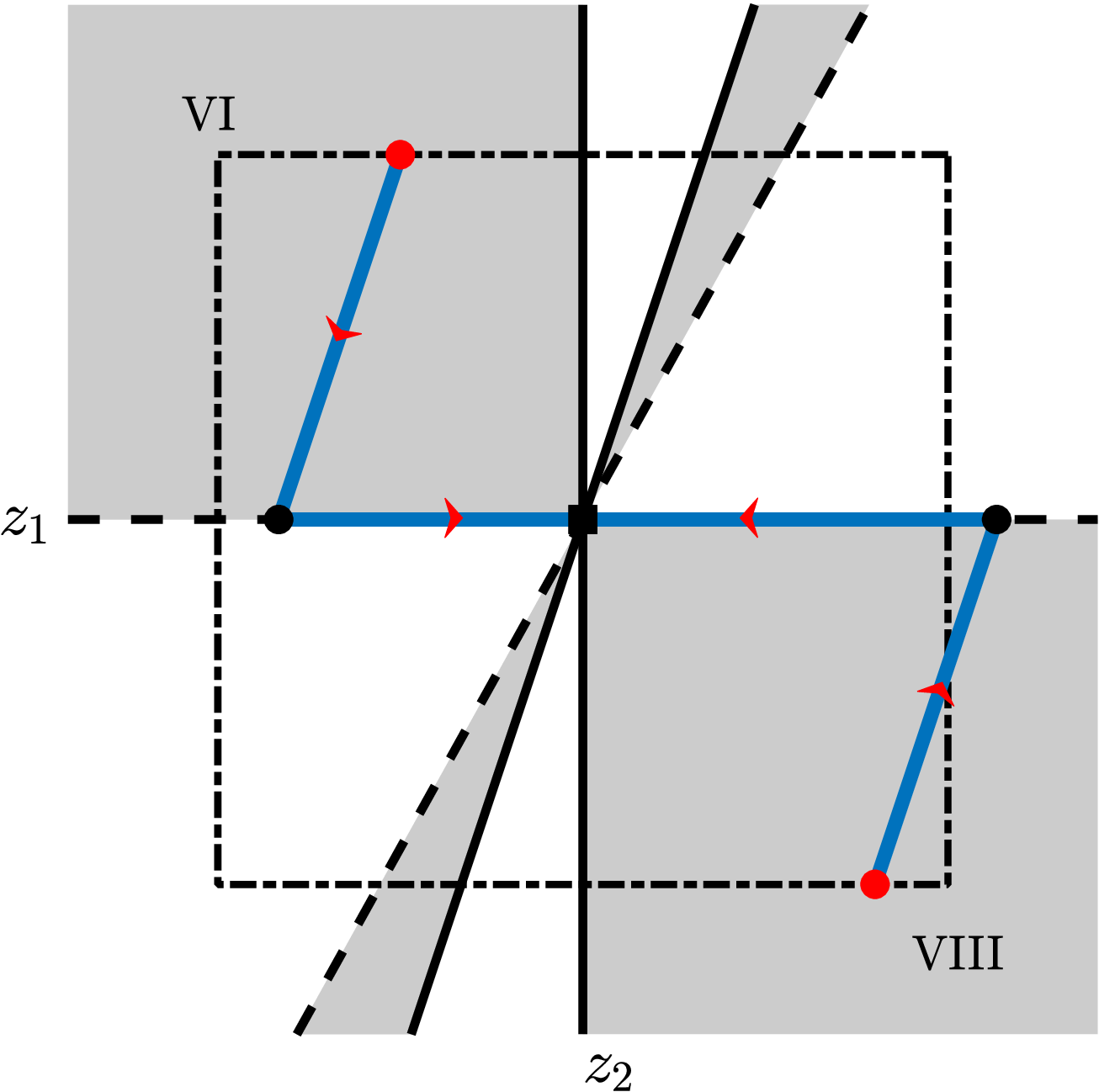} \hspace{7mm}
\includegraphics[width=0.42\textwidth]{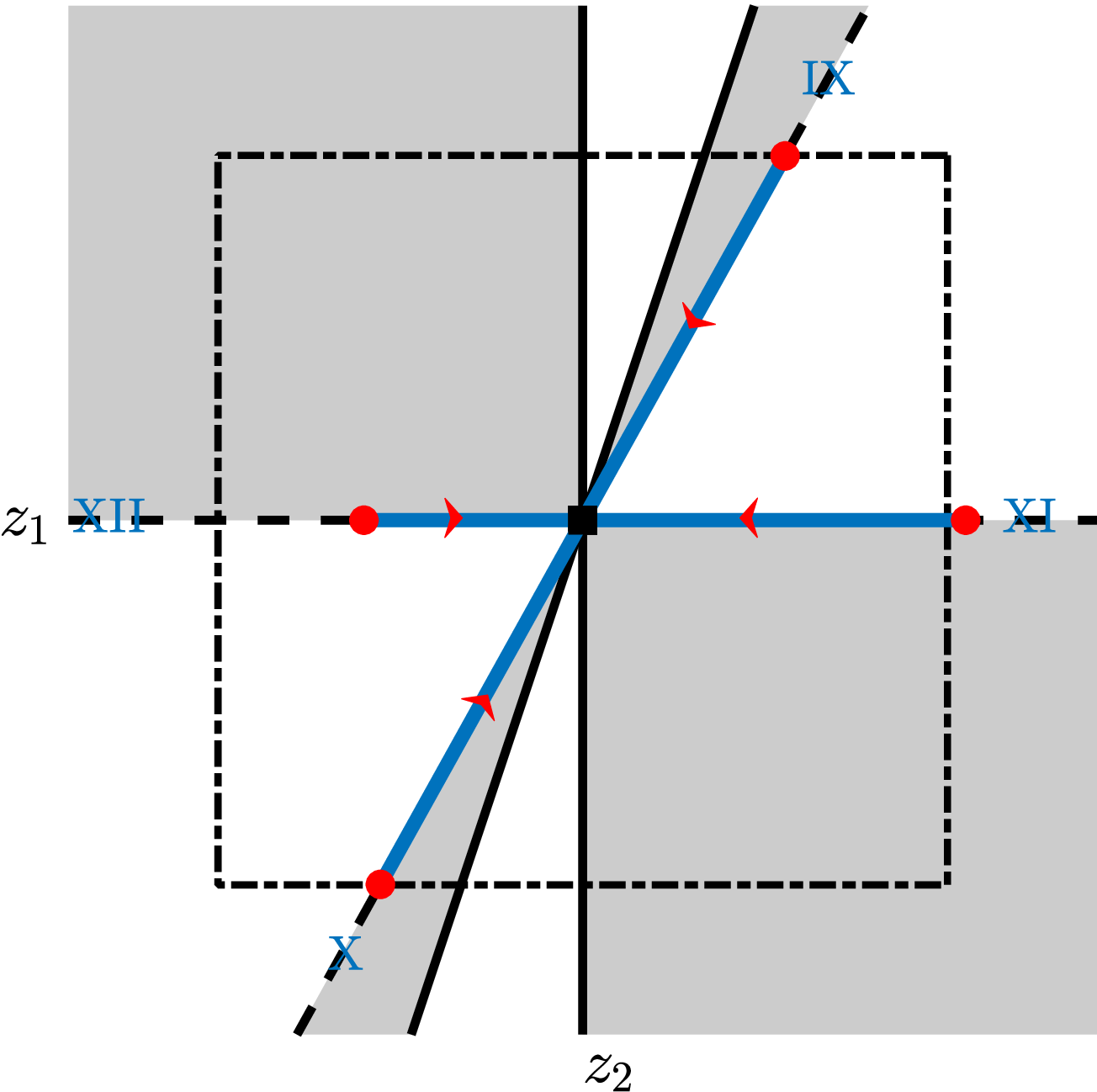}
\caption{\small\sf Starting in Regions {IV}, VIII (left) and IX, X, XI, XII (right).}
\label{fig:h1h2_2_4}
\end{subfigure}
\caption{\sf $z_2$ vs $z_1$ for the example boundary conditions given for the regions listed in~Table~\ref{table:r1r2Dual}. The dashed-dotted square drawn in each diagram has the top-right corner coordinate $(1,1)$.}
\label{fig:subfiguresz1z2}
\end{figure}

Each trajectory $z(c)$ is a concatenation of linear segments, as shown in Figure~\ref{fig:subfiguresz1z2}, except for the case when $(r_1,r_2)=(0,0)$, in which the trajectory $z(c)$ is a single point at the origin, for all $c>0$.  We explain a particular trajectory of Figure~\ref{fig:h1h2_1_3} in the following example, whose starting point is in Region I. Recall that the definitions of Regions~I--XIII are given in Table~\ref{table:r1r2}.

The example trajectory we want to elaborate on starts from the (red) solid circular point indicated by \circled{1}, at the coordinates $(r_1,r_2) = (1,1)$.  This, in particular, is the case where $v_0=1$, $v_f=0$, $s_0=0$, and $s_f=0$, the solution for which is shown in Figure~\ref{fig:solnUnconstr}. The initial part of the trajectory $z(c)$ is a linear segment that arrives at the black switching point \circled{2}: $\left(r_1-\dfrac{5}{9}r_2,0\right)$, hitting the $z_1$-axis. The trajectory finally approaches, moving along the $z_1$-axis, to the terminal point \circled{3}, the origin $(0,0)$, marked with a black square. 

We include in Table~\ref{table:ztrajectoryDescribe} the coordinates of the starting, switching, and terminal points of each trajectory according to the region in which the initial point $(r_1,r_2)$ lies. We have drawn a square centred at the origin using dashed-dotted lines, with the coordinates of the top-right corner $(1,1)$, to give an idea of the scale of the trajectories for the chosen examples.

\begin{table}[t!]
\centering
{\small
\begin{tabular}{r| l}
Region  & Trajectory of $z(c)$ as line segments between indicated points \\[1mm]
\hline\\[-3mm]
 I, V  &  $(r_1,r_2)\longmapsto (r_1-5r_2/9,0)\longmapsto (0,0)$ \\[1mm]
 II, VI  &  $(r_1,r_2)\longmapsto (0,-9r_1/5+r_2)\longmapsto (0,0)$ \\[1mm]
 III, VII  &  $(r_1,r_2)\longmapsto (0,-9r_1/5+r_2)\longmapsto (0,-3r_1+r_2)\longmapsto (r_1-r_2/3,0)\longmapsto (0,0)$ \\[1mm]
 IV, VIII  &  $(r_1,r_2)\longmapsto (r_1-r_2/3,0)\longmapsto (0,0)$ \\[1mm]
 IX, X, XI, XII  &  $(r_1,r_2)\longmapsto (0,0)$ \\[1mm]\hline
\end{tabular}
}
\caption{\sf Key junction points along the trajectory of $z(c)$.}
\label{table:ztrajectoryDescribe}
\end{table}

We specify the ranges of $c$ for each segment of the trajectory $z(c)$ in Table~\ref{table:segment_z}. We also write the parametric expressions for $z(c)$ according to the nine cases from the proof of Theorem~\ref{th:dualCarUnconstrIFF} in Table~\ref{table:z1z2parametricP1}, and include the lines (or trajectories) in the $z_1z_2$-plane in Table~\ref{table:z1z2P1}.

Take again the example in Figure~\ref{fig:h1h2_1_3}, with the starting point $(1,1)$, labelled \circled{1}, in Region I. As $c$ increases, one gets the trajectory segment \circled{1} $\to$ \circled{2} for $c\in[0,2r_2/3)$, which is segment 1, Region I, in Table~\ref{table:segment_z}. As $c$ increases from $2r_2/3$, one gets the trajectory segment \circled{2} $\to$ \circled{3} for $c\in[2r_2/3,12r_1-6r_2]$, which is segment 2, Region I, Table~\ref{table:segment_z}. The trajectory ends at \circled{3}, when $c=12r_1-6r_2$.

\begin{table}[t!]
\centering
{\small
\begin{tabular}{c | l| c| l} 
Region & Trajectory line segments & Region & Trajectory line segments\\
  with &  of $z(c)$ & with &  of $z(c)$ \\
   $z(0)$ & &  $z(0)$ &\\
 [1mm] \hline &&&\\
 I
 & 
$
\begin{array}{ll}
 1: & 0\le c < 2r_2/3 \\
 2: & 2r_2/3\le c \le 12r_1-6r_2
 \end{array} $
&
V
 &
  $\begin{array}{ll}
 1: & 0\le c<-2r_2/3\,,\\
 2: & -2r_2/3\le c\le-12r_1+6r_2 
 \end{array} $
 \\ &&&\\
 \hline &&&\\
II
 &

$\begin{array}{ll}
 1: & 0\le c<6r_1/5 \\
 2: & 6r_1/5\le c \le -6r_1+4r_2
 \end{array} $
&
VI
 &
  $\begin{array}{ll}
 1: & 0\le c<-6r_1/5\,,\\
 2: & -6r_1/5\le c\le6r_1-4r_2 
 \end{array} $
 \\ &&&\\
 \hline &&&\\
 III
 &
$\begin{array}{ll}
 1: & 0\le c<6r_1/5 \\
 2: & 6r_1/5\le c\le 6r_1 \\
 3: & 6r_1<c<2r_2 \\
 4: & 2r_2\le c \le -12r_1+6r_2
 \end{array} $
&
VII
 &
  $\begin{array}{ll}
 1: & 0\le c<-6r_1/5\,,\\
 2: & -6r_1/5\le c\le-6r_1\,, \\
 3: & -6r_1<c<-2r_2\,, \\
 4: & -2r_2\le c\le 12r_1-6r_2 
 \end{array} $
 \\ &&&\\
 \hline &&&\\
IV
 &
  $\begin{array}{ll}
 1: & 0\le c<2r_2 \\
 2: & 2r_2\le c\le-12r_1+6r_2 
 \end{array} $
&
VIII
 &
  $\begin{array}{ll}
 1: & 0\le c<-2r_2\,, \\
 2: & -2r_2\le c\le12r_1-6r_2 
 \end{array} $
 \\ &&&\\
 \hline\hline &&&\\
 IX
 &
 $\begin{array}{ll}
 1: & 0\le c\le6r_1/5 
 \end{array} $
 &
 X
 &
 $\begin{array}{ll}
 1: & 0\le c\le-6r_1/5 
 \end{array} $
 \\ &&&\\
 \hline &&&\\
  XI
 &
 $\begin{array}{ll}
 1: & 0\le c\le12r_1-6r_2 
 \end{array} $
 &
  XII
 &
 $\begin{array}{ll}
 1: & 0\le c\le-12r_1+6r_2 
 \end{array} $
 \\[3mm]
   \hline
 \end{tabular}
 }
\caption{\sf The range of $c$ for each line segment of the trajectory $z(c)$ whose starting regions are as in Table~\ref{table:r1r2}.}
\label{table:segment_z}
\end{table}

\begin{table}[t!]
        \centering
        {\small
 \begin{tabular}{c| l| c| l}
  Case  &  $z(c) = (z_1(c),z_2(c))$  & Case  &  $z(c) = (z_1(c),z_2(c))$\\ [1mm]
\hline &&&\\[-3mm]
 1  &  $\left(-5c/6+r_1,\ -3c/2+r_2\right)$  &  4  &  $\left(5c/6+r_1,\ 3c/2+r_2\right)$ \\
  &&&\\
  2  &  $\left(c/6+r_1,\ c/2+r_2\right)$ & 3  &  $\left(-c/6+r_1,\ -c/2+r_2\right)$ \\
  &&&\\
   5  &  $\left(0,\ -c/4-3r_1/2+r_2\right)$  &  6  &  $\left(0,\ c/4-3r_1/2+r_2\right)$  \\
  &&&\\
    7  &  $(-c/12+r_1-r_2/2,\ 0)$ & 8  &  $(c/12+r_1-r_2/2,\ 0)$ \\ 
  &&&\\
     9  &  $(0,\ 0)$ & & \\[1mm]\hline
 \end{tabular}
 }
\caption{\sf $(z_1(c),\,z_2(c))$ for Cases 1--9.}
\label{table:z1z2parametricP1}
\end{table}
 
\clearpage

\begin{table}[t!]
\centering
{\small
\begin{tabular}{c| c| c| c| c| c}
Case  & Eqn of line & Case  & Eqn of line  & Case  & Eqn of line \\ [1mm]
\hline &&&&\\[-3mm]
1 \& 4  &  $z_2 = 9z_1/5-9r_1/5+r_2$  & 5 \& 6  &  $ z_1 = 0$ & 9  &  $(z_1,z_2) = (0,0)$\\
&&&&\\
2 \& 3  &  $ z_2 = 3z_1-3r_1+r_2$  &  7 \& 8  &  $ z_2 = 0$ & & \\[1mm]\hline
\end{tabular}
}
\caption{\sf Lines in the $z_1z_2$-plane.}
\label{table:z1z2P1}
\end{table}

Interpretations of various other trajectories listed in Table~\ref{table:ztrajectoryDescribe} and depicted in Figure~\ref{fig:subfiguresz1z2} can be found in~\cite{Liu2022}.

The trajectory $z(c)$ shows a correspondence relation between the change in the signs of the constraint functions and the structure of the dual functions. Consider again the example problem with the boundary condition values $s_0= 0$, $s_f= 0$, $v_0= 1$, and $v_f= 0$, i.e., row 1, column 3 in Table~\ref{table:r1r2Dual}. It is the one illustrated in Region I, Figure~\ref{fig:h1h2_1_3} left. In this figure, we can see the changes in the signs of the constraint functions $[h_1(c),h_2(c)]$: for the section \circled{1} -- \circled{2}, we have $h_1>0$ and $h_2>0$; for the section \circled{2} -- \circled{3}, we have $h_1>0$ and $h_2=0$; and at \circled{3} - the origin, we have $h_1=0$ and $h_2=0$. Using Table~\ref{table:r1r2Dual}, we see that the dual function of this problem has three cases: $q^{(1)}$, $q^{(7)}$ and $q^{(9)}$, corresponding to Cases 1, 7 and 9, respectively, in Theorem~\ref{th:dualCarUnconstrIFF}. These cases match the three situations we discussed previously in Region I, Figure~\ref{fig:h1h2_1_3} (left).

\subsection{The PDP iterates of Problem~(P) under different choices of the step-size}
\label{sec:iterates_p1}

We will now elaborate on the outcome of our experiments using four different choices of the step-size, two of which are proposed in the present paper.  We consider again the boundary values $s_0= 0$, $s_f= 0$, $v_0= 1$, and $v_f= 0$, shown in row 1 and column 3 of Table~\ref{table:r1r2Dual}. In Figure \ref{fig:step-sizeDual12}, we show the iterations of $c$ against the dual function $q(c)$ using the step-size as in \eqref{sk_DSG1} in Figure \ref{fig:step-sizeDual1} and the step-size as in \eqref{sk_DSG2} in Figure \ref{fig:step-sizeDual2}. The solid (blue) curve is the graph of the dual function $q(c)$ and the solid (red) dots are the iterates generated by the PDP algorithm.\\
In these experiments, we have set the PDP algorithmic parameters $\alpha_k=0.5$, $c_0=0.1$. With the step-size as in \eqref{sk_DSG1}, we have set $\eta=0.1$, $\beta=4$, and we have taken $s_k=0.5\eta_k+0.5\beta_k$. In the step-size as in \eqref{sk_DSG2}, we have set $\theta_k=\beta=1$, which results in $s_k = 1/\|h(u_k)\|_1$ and so $c_{k+1} = c_k + \tilde{s}_k\,\|h(u_k)\|_1 = c_k + 1.5$, since $\tilde{s}_k = 1.5\,s_k$ and $s_k = 1/\|h(u_k)\|_1$. It takes $38$ iterations with the step-size in~\eqref{sk_DSG1}  to find the solution and just $4$ iterations with the step-size in \eqref{sk_DSG2}. With the  step-size in \eqref{sk_DSG1}, most iterations are performed when the dual value $q(c_k)>1.99$ ($37$ out of $38$ iterations), with the dual optimal value  $\bar q=2$ \textcolor{black}{ when $c>5.9$ ($32$ out of $38$)}.
 \begin{figure}[t!]
  \centering
  \begin{subfigure}{.49\textwidth}
\centering
  \includegraphics[width=\textwidth]{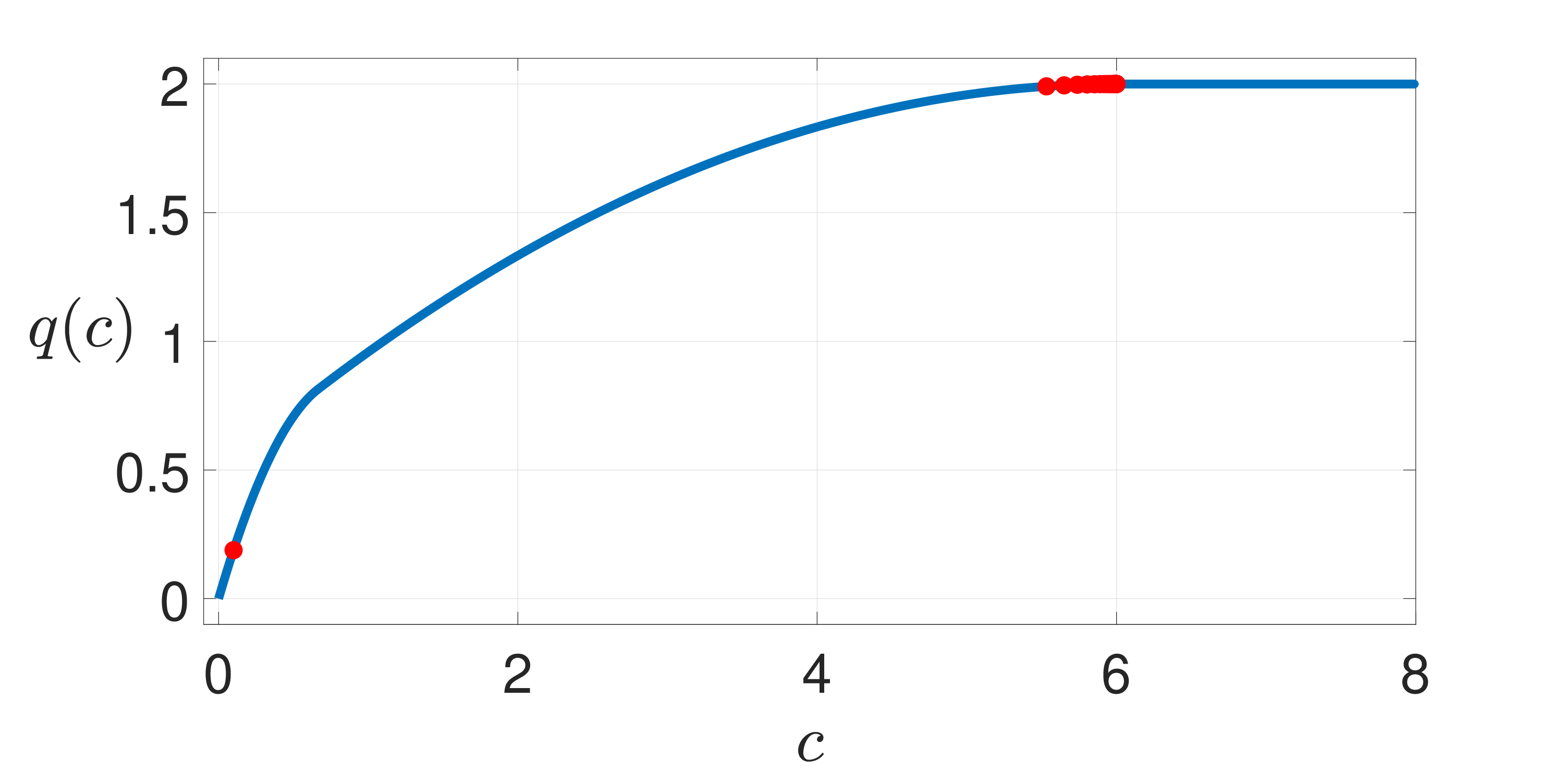}
  \caption{\small Algorithm~PDP with $s_k$ in \eqref{sk_DSG1}.}
  \label{fig:step-sizeDual1}
\end{subfigure}
\begin{subfigure}{.49\textwidth}
 \centering
  \includegraphics[width=\textwidth]{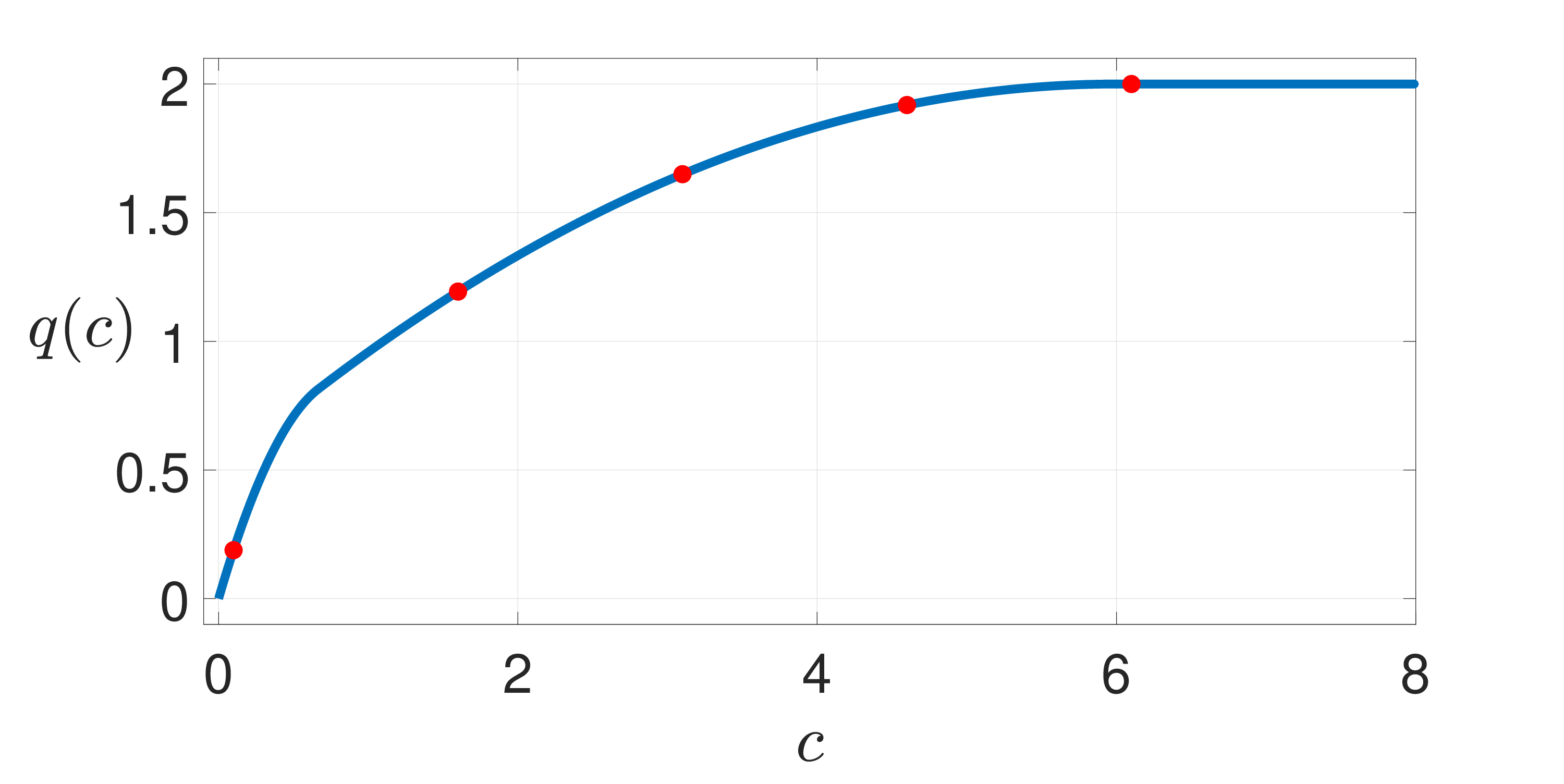}
  \caption{\small Algorithm~PDP with $s_k$ in \eqref{sk_DSG2}.}
  \label{fig:step-sizeDual2}
  \end{subfigure}
  \caption{\sf The dual function updates (shown by red dots on the blue curve representing the graph of the dual function) in each iteration of Algorithm~PDP using step-sizes of types 1 and 2.}
  \label{fig:step-sizeDual12}
\end{figure}

If we adjust $s_k$ slightly to $s_k=0.4\eta_k+0.6\beta_k$, then we can obtain the solution in just one iteration. We also include the first four iterations of $c$ for both choices of step-size as in \eqref{sk_DSG1} and \eqref{sk_DSG2} in Table~\ref{table:cUpdatesPDP12}.  Note that the calculations have been performed analytically; therefore, the iterations displayed in the table are exact to the number of decimal places shown.
\begin{table}[t!]
\centering
{\small
\begin{tabular}{l| l r | l r } 
  & \multicolumn{2}{c|}{step-size of type 1}& \multicolumn{2}{c}{step-size of type 2} \\
  \cline{2-3} \cline{4-5}
 $k$ & $c_k$ & $q(c_k)$ & $c_k$ & $q(c_k)$\\[1mm]
 \hline &&&& \\[-3mm]
 0 & 0.1 & 0.188333 & 0.1 & 0.188333\\ &&&&\\
 1 & 5.532500 & 1.990893 & 1.6 & 1.193333\\ &&&&\\
 2 & 5.650513 & 1.994910 & 3.1 & 1.649583\\ &&&&\\
 3 & 5.738521 & 1.997512 & 4.6 & 1.918333\\ &&&&\\
 4 & 5.804247 & 1.998403 & 6.1 & 2.000000\\ [1mm]
 \hline
 \end{tabular}
 }
\caption{Updates of $c$ for algorithm PDP under step-sizes of types 1 and 2.}
\label{table:cUpdatesPDP12}
\end{table}

We plot the function iterates $u_k$ in Figure~\ref{fig:uConverge} under the step-size in \eqref{sk_DSG2}.  Here, $u_k$ are the minimizers of the dual function $q(c_k)$ (for $k=0,1,2,3,4$). Although in the initial iteration ($k=0$), $u_0$ seems to be a horizontal line, that is not the case. In fact, the initial function iterate is a linear function with slope $0.1$.  It is interesting to observe that, apart from the initial iteration, all iterates $u_k$, $k=1,2,3,4$,  intersect at the (same) point $\left(1/2,-1\right)$, which can easily be computed analytically: By Corollary~\ref{cor:boundary01_p1} the initial iteration, which corresponds to $c\in\left[0,2/3\right)$, gives $u(t)=ct-2c$, and all other four iterations fall into the case when $c\in\left[2/3,6\right]$, which yields $ u(t)= ct-(c+2)/2$. From the expression of $u$, we trivially conclude that they all cross the point $\left(1/2,-1\right)$.

\begin{figure}[t!]
  \centering
    \includegraphics[width=0.7\textwidth]{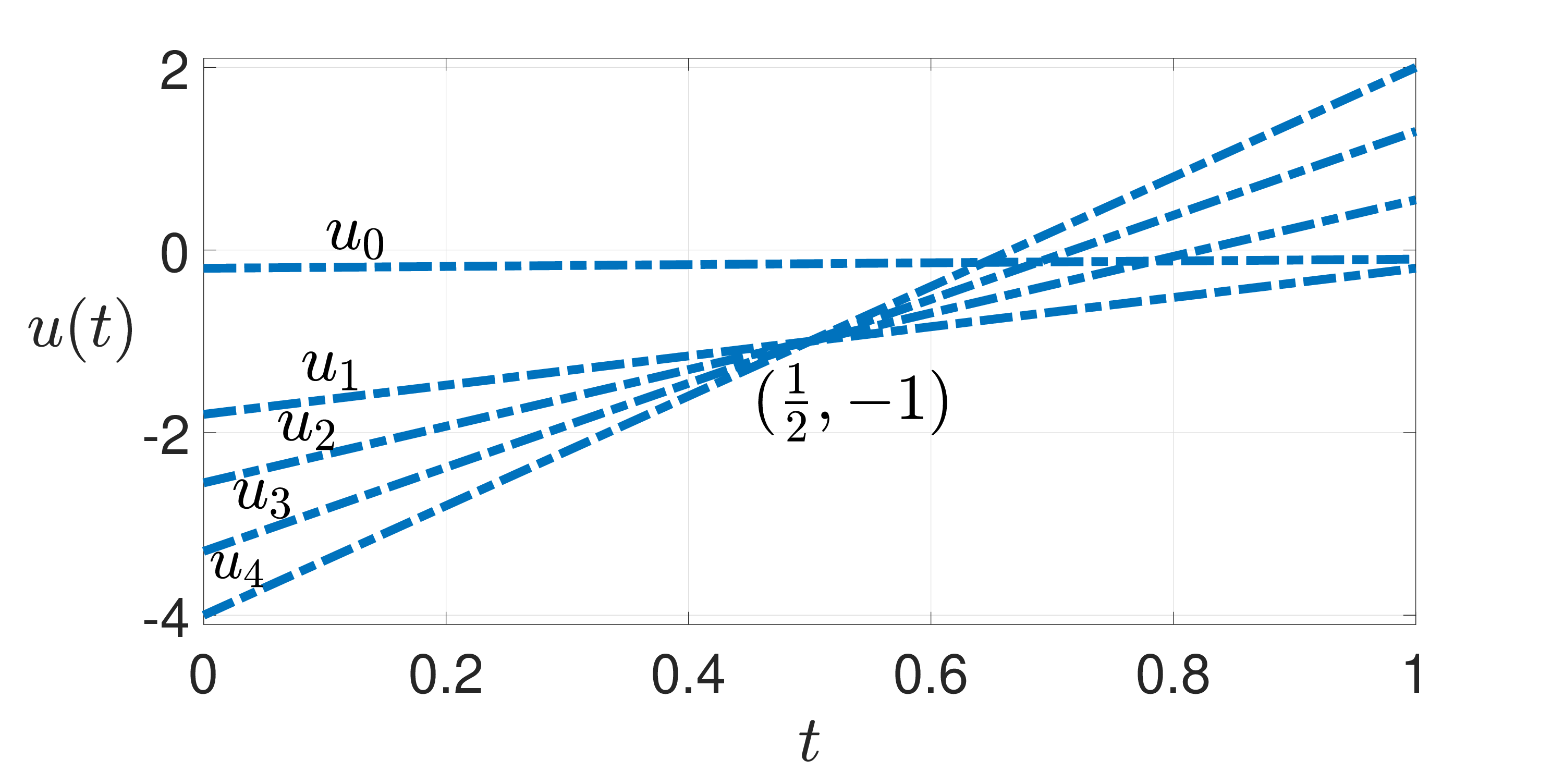}
    \caption{\sf The iterations of $u(t)$ under algorithm PDP with step-size of type 2.}
  \label{fig:uConverge}
\end{figure}

\subsection{Comparison with the Polyak-type step-size}\label{sec: PolyakComparison}

For comparison purposes, we have also implemented two different Polyak-type step-sizes~\cite{Erm1966,Pol1967} in our experiments. See the iterations of the values of the dual function superimposed on the graph of the dual function $q$ vs the penalty parameter $c$ in Figure~\ref{fig:Polyak}. We firstly use the step-size 
\begin{equation}\label{PolyakStep1}
   s_k= \dfrac{\bar q - q_k}{\sigma_k} 
\end{equation}
and plot the updates in Figure~\ref{fig:Polyak1}, where $q_k:=L(u_k,c_k)$ is the current dual value, $\sigma_k:=|h_1(u_k)|+|h_2(u_k)|$, and $\bar q$ is the optimal dual value, which as mentioned above is equal to $2$ for this example. Figure \ref{fig:Polyak2} shows the updates using a second step-size: 
\begin{equation}\label{PolyakStep2}
   s_k= \dfrac{(1+\alpha)\delta}{1+(1+\alpha)^2}\left[\dfrac{\hat q-q_k}{\sigma_k}+\alpha(\hat c-c_k)\right],
\end{equation}
introduced in formula $(21)$ of Reference \cite{BKdsg}, where $\hat q$ is an upper bound of the dual value, taking a value $2.1$ here, $\hat c$ takes $6.5$, and both $\delta$ and $\alpha$ take $1$. 

We observe, arguably for the first time in the literature, that the step-size in \eqref{PolyakStep1}, a ``pure form'' of the Polyak-type step-size, results in the penalty parameter update formula to be a {\em Newton iteration}~(see \cite{NW2006,Fle2013} for more information on Newton's method) in solving the equation $\bar{q} - q_k=0$.  One can verify a well known geometric property of Newton's method on the graph in Figure~\ref{fig:Polyak1}: The solid black lines are tangent to the graph of the dual function $q$ at each iteration $c_k$. Each tangent line intersects the optimal dual value ($2$, here), as expected from a Newton iteration, producing the next iterate $c_{k+1}$.

We consolidate our interpretation of the penalty parameter update as a Newton iteration in the following remark.

\begin{remark}[the Newton method and the PDP Algorithm]\hfill\break
\label{rem:Newton_PDP_P1}\rm
Note that the iterations we observe in Figure~\ref{fig:Polyak1} are reminiscent of the Newton method, applied to finding a zero of the function $q(c) - 2 = 0$.  Recall that the Newton iterations are given by
\[
c^{k+1} = c^k - \frac{q(c^k) - 2}{q'(c^k)}\,,\ \ \ k = 0,1,\ldots\,.
\]
The graph implies that the first derivative $q'(c) = 0$.  To see why this holds, recall that $q(c) = L(\bar{u},c)$, where $\bar{u}$ is the minimizer of $L(\cdot,c)$.  So, we get $q'(c) = dL(\bar{u},c)/dc = |h_1(\bar{u})|+|h_2(\bar{u})| = \sigma(\bar{u})$.  Since $\sigma(\bar{u}) = 0$ at the solution, $q'(c) = 0$ at the solution. Therefore, quadratic convergence cannot be guaranteed by the standard convergence theory of the Newton method, which requires $q'(c) \neq 0$ in a neighbourhood of the solution -- see \cite[Theorem 2.4.3]{DenSch1983}.  This explains the ``sluggish'' (less than quadratic convergence) behaviour of the iterations near the solution $c^* \ge 2$.  We point out that our algorithm finds the solution in $23$ iterations using this step-size.  
\endproof
\end{remark}

\begin{figure}[t!]
  \centering
  \begin{subfigure}{0.49\linewidth}
\centering
  \includegraphics[width=\linewidth]{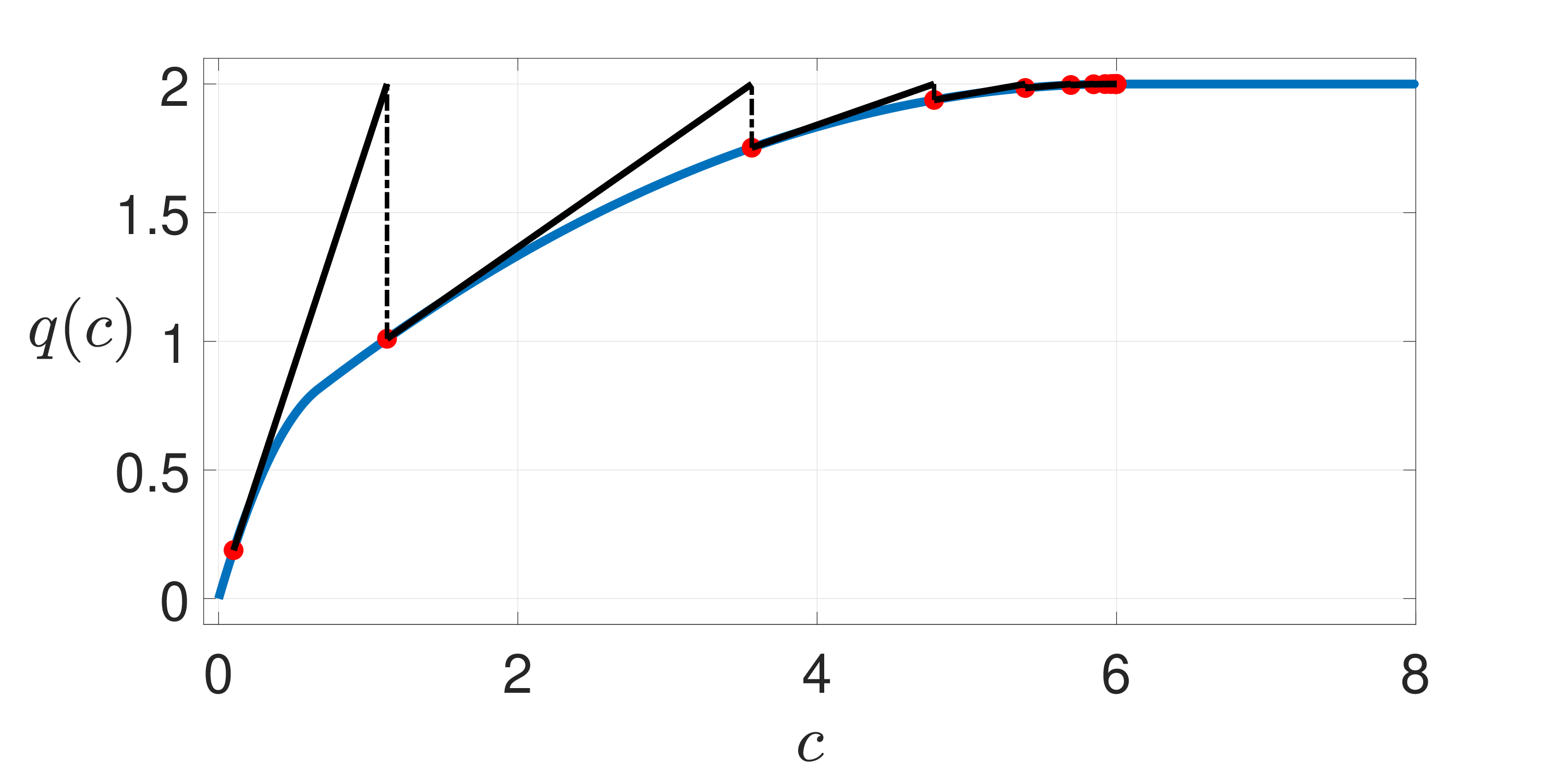}
  \caption{\small Algorithm~PDP with the Polyak-type step-size in~\eqref{PolyakStep1}.}
  \label{fig:Polyak1}
\end{subfigure}
\begin{subfigure}{0.49\linewidth}
 \centering
  \includegraphics[width=\linewidth]{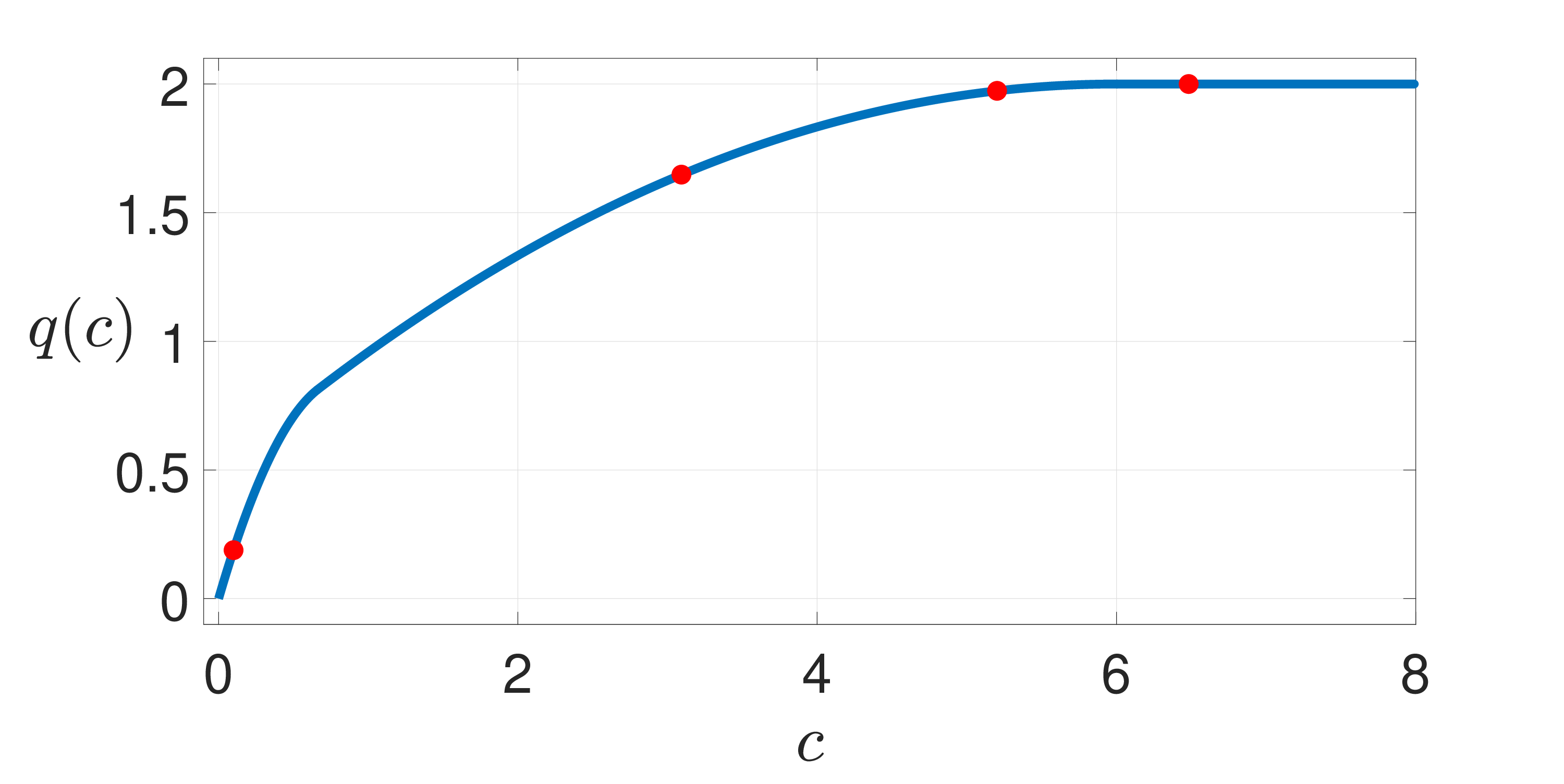}
  \caption{\small Algorithm~PDP with the Polyak-type step-size in~\eqref{PolyakStep2}.}
  \label{fig:Polyak2}
  \end{subfigure}
  \caption{\sf The dual function updates (shown by red dots on the blue curve representing the graph of the dual function) in each iteration of Algorithm~PDP using step-sizes of Polyak type.}
  \label{fig:Polyak}
\end{figure}

\section{Conclusion and Discussion}
We solved the double integrator optimal control problem
via an augmented Lagrangian duality scheme that can be interpreted as a penalty method. To consider all possible instances of the double integrator problem, we divided the problem space into 13 disjoint classes according to the boundary values. For each of these 13 classes, we obtained analytical expressions of (i) the dual function, (ii) the control variable, and the (iii) least exact penalty parameter.

We graphically compared the iterations of the dual sequence using four different choices of step-sizes: \eqref{sk_DSG1}, \eqref{sk_DSG2}, and two choices of Polyak type step-size.

To understand the behaviour of the method, we plotted the trajectory of the constraint functions as a parametric equation involving the penalty parameter $c$. We showed the correspondence between the switchings of the trajectory and the changes of the dual function.

\setcounter{section}{0}

\renewcommand{\thesection}{Appendix}

\renewcommand{\thesection}{A\Alph{section}}
\section*{Appendix} \label{sec:proofs}

\setcounter{equation}{0}
\renewcommand{\theequation}{A.\arabic{equation}}

\textbf{\hypertarget{prf:LEMsubgrad_p1}{Preliminaries to the Proof of Lemma \ref{lem:subgrad_p1}}} \\[1mm]
To establish Lemma \ref{lem:subgrad_p1}, we need {\bf Remark A.1} and {\bf Lemma A.1}, given next.

\begin{remark}[Derivative of a linear functional]
\label{rem:Hilbert}
\rm Recall that the space ${\cal L}^2([0,1];\R)$ is a Hilbert space with scalar product given by
\[
\langle u,u'\rangle= \int_0^1 u(s)\,u'(s) ds,
\]
for every $u,u'\in {\cal L}^2([0,1];\R)$. By~\cite[Theorem 5.5]{Brezis} (Riesz--Frechet representation theorem in Hilbert spaces), for every linear functional $\Psi:{\cal L}^2([0,1];\R) \to \R$, there exists a unique element $w\in {\cal L}^2([0,1];\R)$ such that 
\begin{equation}\label{eq:der}
   \Psi(v)= \langle w,v\rangle= \int_0^1 v(s) w(s) ds,\,\, \forall \,\, v\in {\cal L}^2([0,1];\R).
\end{equation}
Moreover, it is easy to deduce that $w=\dfrac{d \Psi}{d v}\in {\cal L}^2([0,1];\R)$ is the derivative of $\Psi$. This follows directly from observing that
\[
\lim_{v\to 0}\frac{\Psi(u+v)-\Psi(u)-\langle w,v\rangle}{\|v\|}
  =\lim_{v\to 0}\frac{\langle w,u+v\rangle - \langle w,u\rangle-\langle w,v\rangle}{\|v\|}
  =0,
\]
where we used \eqref{eq:der}. The above expression implies that $w$ is the Frechet derivative of $\Psi$.
\end{remark}

\begin{lemma}[Derivative of the double integrator] \label{lem:A1}
Consider the function\hfill\\ $\Lambda_1:{\cal L}^2([0,1];\R)\to \R$ defined by
\[
\Lambda_1(u)=\ds \int_0^1 \left[ \int_0^t u(\tau) d\tau\right] dt, 
\]
where $u\in {\cal L}^2([0,1];\R)$. Define $\gamma:[0,1]\to [0,1]$ as $\gamma(s):=1-s$ for every $s\in [0,1]$.
Then $\Lambda_1$ is linear and we have $\dfrac{d \Lambda_1}{d u} = \gamma$.
\end{lemma}
\begin{proof}
The linearity of $\Lambda_1$ is directly related to the definition and properties of the integral. By Remark \ref{rem:Hilbert}, to prove that $\dfrac{d \Lambda_1}{d u} = \gamma$, it is enough to  show that
\begin{equation}\label{eq:Lem-L1}
    \Lambda_1(u) = \langle u,\gamma \rangle=  \int_0^1 u(s) \gamma(s) ds,\,\, \forall \,\, u\in {\cal L}^2([0,1];\R).
\end{equation}
Indeed, for a fixed $t\in [0,1]$ define ${\mathcal\chi}_{(0,t]}:[0,1]\to \R$, as 
\[\chi_{(0,t]}(\tau):=
  \left\{
    \begin{array}{ll}
      1, \quad &  \tau\le t, \\
      0, \quad &  \tau>t,\\
    \end{array}
  \right. 
  \]
  for every $\tau\in [0,1]$. Similarly, for a fixed $\tau\in [0,1]$, we define
${\mathcal\chi}_{(0,\tau)}:[0,1]\to \R$, as 
\[\chi_{(0,\tau)}(t):=
  \left\{
    \begin{array}{ll}
      1, \quad &  \tau< t, \\
      0, \quad &  \tau\ge t,\\
    \end{array}
  \right. 
  \]
  for every $t\in [0,1]$. It is easy to see that $\chi_{(0,t]}, \chi_{(0,t)}\in {\cal L}^2([0,1];\R)$ (with $\|\chi_{(0,t]}\|_2=\|\chi_{(0,t)}\|_2=\sqrt{t}\le 1$). Using the definitions, we can write
\[
\begin{array}{rcl}
    \Lambda_1(u)   & =  &\ds \int_0^1 \left[ \int_0^t u(\tau) d\tau\right] dt=\int_0^1\int_0^1\chi_{(0,t]}(\tau)u(\tau)d\tau dt\\
    &&\\
    &=&\ds\int_0^1\left[\int_0^1\chi_{(0,t]}(\tau)u(\tau)dt\right] d\tau =\int_0^1u(\tau)\left[\int_0^1\chi_{(0,t]}(\tau)dt\right]d\tau\\
    &&\\
   & =&\ds\int_0^1u(\tau)\left[\int_0^1\left(1-\chi_{(0,\tau)}(t)\right)dt\right] d\tau
    \\
    &&\\
   & =&\ds\int_0^1u(\tau)\left[1-\int_0^\tau 1 dt\right] d\tau =\int_0^1u(\tau)(1-\tau)d\tau,
\end{array}
\]
where we have changed the order of integration in the third equality, and in the fifth equality we used the fact that  $\chi_{(0,t]}(\tau)=1-\chi_{(0,\tau)}(t)$ for every $\tau,t\in [0,1]$. The above expression establishes \eqref{eq:Lem-L1}.
\end{proof}

We can now proceed with the {\bf proof of Lemma} \ref{lem:subgrad_p1}.

\begin{proof} 
 $(i)$  
With the notation of Lemma \ref{lem:A1}, we consider $\Lambda_1(u) = \ds\int_0^1\left[\int_0^t u(\tau)\,d\tau\right]$.  Then
\begin{equation}
 \begin{array}{rcl}
h_1(u)&=&\Lambda_1(u) +r_1\,,\quad h_2(u)=\ds\int_0^1u(t)dt +r_2\,,\quad \varphi(u)=\ds\frac{1}{2}\int_0^1u^2(t)dt\,, \label{test}
\end{array}
\end{equation}
where $r_1,r_2$ are as in \eqref{eq:r1_r2} and $h_1,h_2$ are as in \eqref{eq:h_p1}. 
By \eqref{eq:LagrnCarUnconstr}, the Lagrangian for Problem~(P) is given by 
\begin{equation} \label{eq:LP1}
    L(u,c)=\varphi(u)+c|h_1(u)|+c|h_2(u)|. 
\end{equation}
Since the functions $h_i$ are affine, they are continuously differentiable and convex, so we can apply\linebreak \cite[Proposition~4.5.1 (sum rule)]{Schirotzek2007} to deduce that
\[
\partial_u L(\cdot,c)=\dfrac{d \varphi}{d u}(\cdot)+c\,  \partial \theta_1(\cdot) +c\,\partial \theta_2(\cdot),
\]
where $\theta_1(u):=|h_1(u)|$, and $\theta_2(u):=|h_2(u)|$. Using now the fact that $h_1$ and $h_2$ are affine, \eqref{eq:comp} gives
\[
 \partial_u L(u,c)=\dfrac{d \varphi}{d u}(u)+c\,\alpha_1(u)\,\frac{d h_1}{d u}(u)+c\,\alpha_2(u)\,\frac{d h_2}{d u}(u),
\]
where $\alpha_i(u)={\rm sign}(h_i(u))$ if $h_i(u)\neq 0$ and $\alpha_i(u)\in [-1,1]$ if $h_i(u)= 0$, for $i=1,2$. Direct differentiation yields $ \dfrac{d \varphi}{d u}(u) = u$. We claim now that $\dfrac{d h_2}{d u} (u)= \xi$, where $\xi(s)=1$ for every $s\in [0,1]$. Indeed, using the definition of $\xi$ we can write
\[
h_2(u)=\int_0^1 u(s) ds+r_2=\int_0^1 u(s) \xi(s)ds+r_2.
\]
Note that $h_2$ is an affine function of $u$, with linear part $\lambda(u):= \int_0^1 u(s)ds$. Hence, its derivative w.r.t $u$ will coincide with the derivative w.r.t $u$ of its linear part $\lambda(u)= \int_0^1 u(s)\xi(s)ds$. Now, the claim follows from Remark \ref{rem:Hilbert} for $\Psi:=\lambda$ and $w:=\xi$. To compute $\dfrac{d h_1}{d u}(u)$, we will apply Lemma~\ref{lem:A1}. Indeed, differentiating the first equation in \eqref{test} and using the notation of Lemma~\ref{lem:A1}, we have
\[
\dfrac{d h_1}{d u}= \dfrac{d \Lambda_1}{d u}=\gamma.
\]
Together,
\[
 \partial_u L(u,c)=u+c\,\alpha_1(u)\, \gamma+c\,\alpha_2(u),
\]
establishing \eqref{minimizerLagrnCarUnconstr}. Note that since $\xi\equiv 1$ in $[0,1]$, we do not write $\xi$ in the third term.
\\[2mm]
$(ii)$ Recall that the dual function $q$ is obtained as follows
\[
q(c):=\min_{u\in H} L(u,c),
\]
where $H={\cal L}^2([0,1];\R)$. Since the objective function $\varphi$ is strongly convex on $u$, the minimizer of the Lagrangian is unique, and we denote it as $u_c$. This minimizer has to verify
\[
0\in \partial_u L(u_c,c).
\]
Using now part (i), there exist $(\alpha_i(u))_{i=1,2}$ such that
\[
0=u_c+c\,\alpha_1(u_c)\, \gamma+c\,\alpha_2(u_c).
\]
Solving for $u_c$, and using the definition of $\gamma$, the above expression becomes 
\[
u_c(t)=c\alpha_1(u_c) t-c(\alpha_1(u_c)+\alpha_2(u_c)).
\]
The proof is complete.
\end{proof}

\textbf{\hypertarget{prf:LEMcrSet_disjoint}{Proof of Lemma~\ref{lm:crSet_disjoint}}}

\begin{proof}
Recall the definitions of $V^i$, $i=1,\ldots ,9$, in~\eqref{eq:r1r2cSet}. We will show for each $i$, $i=1,\ldots ,8$ that $V^i\cap V^j =\emptyset$, for $j=i+1,\ldots ,9$. 

$i=1.$ We will show that $V^1\cap V^k=\emptyset$ for $k=2,\ldots,9$. From the definition of $V^1$ in~\eqref{eq:r1r2cSet}, we know that $r_1,r_2>0$, which trivially yields that $V^1$ does not intersect $V^2$, $V^4$ or $V^6$. By the definition of $V^1$, we have $c<\min\left\{\dfrac{6}{5}r_1,\dfrac{2}{3}r_2\right\}$, which implies that $V^1\cap V^j=\emptyset$ for $j=3,5,7,8$. Next, we show that $V^1\cap V^9=\emptyset$. Indeed, assume that $(r_1,r_2,c)\in V^1\cap V^9$. We have from the definition of $V^1$ and $V^9$, the following two simultaneous inequalities
\begin{equation}\label{eq:cr_V1V9}
c<\dfrac{6}{5}r_1\quad \hbox{ and }\quad c\ge\max\{|12r_1-6r_2|,|-6r_1+4r_2|\}.
\end{equation}
We will show that the two inequalities in \eqref{eq:cr_V1V9} contradict each other and thus $V^1\cap V^9=\emptyset$. Since $(r_1,r_2,c)\in V^1\cap V^9$, we have that $(r_1,r_2)$ is in $S_{I}\cup S_{IX}$, $S_{II}$ or $S_{III}$, giving rise to three possible situations. For each of the three situations, we now show that the second inequality in \eqref{eq:cr_V1V9} yields 
\begin{equation}\label{eq:cr_V9_1}
c\ge\dfrac{6}{5}r_1,
\end{equation}
 in contradiction with the first inequality in \eqref{eq:cr_V1V9}. Assume that we are in the first situation, i.e., $(r_1,r_2)\in S_{I}\cup S_{IX}$. The second inequality in \eqref{eq:cr_V1V9} gives $c\ge 12r_1-6r_2$. By the definition of $S_I$ and $S_{IX}$, $r_2\le\dfrac{9}{5}r_1$. The last two inequalities yield 
 \[
 c\ge 12r_1-6r_2\ge 12r_1-6 \left(\dfrac{9}{5}r_1\right)=\dfrac{6}{5}r_1,
 \]
 where we used the fact that $r_1>0$. Hence  \eqref{eq:cr_V9_1} holds, so this situation cannot occur. Assume now that $(r_1,r_2)\in S_{III}$. The second inequality in \eqref{eq:cr_V1V9} implies $c\ge -12r_1+6r_2$. By the definition of $S_{III}$ we must have $r_2>3r_1$. Using the  last two inequalities, we obtain the following.
 \[
 c\ge -12r_1+6r_2> -12r_1+6 \left(3r_1\right)=6r_1\ge \dfrac{6}{5}r_1,
 \]
where we again used the fact that $r_1>0$. Hence \eqref{eq:cr_V9_1} holds, and this case cannot happen. Finally, assume that $(r_1,r_2)\in S_{II}$. The second inequality in \eqref{eq:cr_V1V9} yields $c\ge-6r_1+4r_2$. By the definition of $S_{II}$, we have $r_2>\dfrac{9}{5}r_1$. Using the last two inequalities, we can write 
  \[
 c\ge -6r_1+4r_2 > -6r_1+4 \left(\dfrac{9}{5}r_1\right)= \dfrac{6}{5}r_1,
 \]
 showing again that  \eqref{eq:cr_V9_1} holds. Together, we must have $V^1\cap V^9=\emptyset$.

$i=2.$ Assume that $(r_1,r_2,c)\in V^2$. The definition of $V^2$ trivially yields $V^2\cap V^3=V^2\cap V^5=V^2\cap V^8=\emptyset$. The definition of $V^2$ also implies that $\max\{0,-6r_1\}<c<-2r_2$, which implies that $V^2$ does not intersect $V^4$, $V^6$ or $V^7$. Next, we show that $V^2\cap V^9=\emptyset$. Assume that $(r_1,r_2,c)\in V^2\cap V^9$.   Since $(r_1,r_2,c)\in V^2\cap V^9$, we must have the following two simultaneous inequalities
\begin{equation}\label{eq:cr_V2V9}
c<-2r_2\quad \hbox{ and }\quad c\ge\max\{|12r_1-6r_2|,|-6r_1+4r_2|\}.
\end{equation}
We will show that the two inequalities in \eqref{eq:cr_V2V9} contradict each other. By the definition of $V^2$, we have two possible cases: either $(r_1,r_2)\in S_{VII}$ or  $(r_1,r_2)\in S_{VIII}$. Assume first that $(r_1,r_2)\in S_{VII}$, this implies that $r_1,r_2<0$ and $r_2<3r_1$.
The second inequality in \eqref{eq:cr_V2V9} yields $c\ge 12r_1-6r_2$. Combining the latter inequality with the fact that $r_2<3r_1$, yields $c>-2r_2$, which contradicts the leftmost inequality in \eqref{eq:cr_V2V9}. This shows that we cannot have $(r_1,r_2)\in S_{VII}$. Assume now that $(r_1,r_2)\in S_{VIII}$. This means that $r_1\ge 0$ and $r_2<0$. The first and second inequality in \eqref{eq:cr_V2V9} yield
\[
-2r_2> c\ge 6r_1-4r_2\ge -4r_2, 
\]
which gives $r_2>0$, a contradiction. Hence we have that $V^9\cap V^2=\emptyset$.

$i=3.$ Assume that $(r_1,r_2,c)\in V^3$. By definition of $V^3$, we have that $r_2>0$, $r_2>3r_1$, which trivially leads that none of $V^4$, $V^6$ or $V^7$ intersects $V^3$. By definition of $V^3$, we also have that $\max\{0,6r_1\}<c<2r_2$, which gives the result that $V^5\cap V^3=\emptyset$, and $V^8\cap V^3=\emptyset$. We show next that $V^3\cap V^9=\emptyset$. Assume that $(r_1,r_2,c)\in V^3\cap V^9$. By the definitions of $V^3$ and $V^9$, we have the following two simultaneous inequalities
\begin{equation}\label{eq:cr_V3V9}
c<2r_2\quad \hbox{ and }\quad c\ge\max\{|12r_1-6r_2|,|-6r_1+4r_2|\}.
\end{equation}
We will show that the two inequalities in \eqref{eq:cr_V3V9} contradict each other. The second inequality in \eqref{eq:cr_V3V9} yields $c\ge -12r_1+6r_2$. Combining the latter inequality with the fact that $r_2>3r_1$ yields $c>2r_2$, which contradicts the leftmost inequality in \eqref{eq:cr_V3V9}. Hence we have $V^9\cap V^3=\emptyset$.

$i=4.$ Assume that $(r_1,r_2,c)\in V^4$. By definition of $V^4$, we have $r_1,r_2<0$, which trivially leads that $V^5\cap V^4=\emptyset$. By definition of $V^4$, we also have $c<\min\left\{-\dfrac{6}{5}r_1,-\dfrac{2}{3}r_2\right\}$, which gives the result that $V^4$ does not intersect $V^6$, $V^7$ or $V^8$. Next, we show that $V^4\cap V^9=\emptyset$. Assume that $(r_1,r_2,c)\in V^4\cap V^9$. By the definitions of $V^4$ and $V^9$, we have the following two simultaneous inequalities
\begin{equation}\label{eq:cr_V4V9}
c<-\dfrac{6}{5}r_1\quad \hbox{ and }\quad c\ge\max\{|12r_1-6r_2|,|-6r_1+4r_2|\}.
\end{equation}
We will show that the two inequalities in \eqref{eq:cr_V4V9} contradict each other and thus $V^4\cap V^9=\emptyset$. Since $(r_1,r_2,c)\in V^4\cap V^9$, we have that $(r_1,r_2)$ is in $S_{V}\cup S_{X}$, $S_{VI}$ or $S_{VII}$, giving rise to three possible situations. For each of the three situations, we now show that the second inequality in \eqref{eq:cr_V4V9} yields 
\begin{equation}\label{eq:cr_V9_4}
c\ge-\dfrac{6}{5}r_1,
\end{equation}
 in contradiction with the first inequality in \eqref{eq:cr_V4V9}. Assume first that $(r_1,r_2)\in S_{V}\cup S_{X}$. The second inequality in \eqref{eq:cr_V4V9} gives $c\ge -12r_1+6r_2$. By the definitions of $S_V$ and $S_{X}$, $r_2\ge\dfrac{9}{5}r_1$. The last two inequalities yield \eqref{eq:cr_V9_4}, so this situation cannot happen. Assume now that $(r_1,r_2)\in S_{VII}$. The second inequality in \eqref{eq:cr_V4V9} implies $c\ge 12r_1-6r_2$. By the definition of $S_{VII}$ we must have $r_2<3r_1$. The last two inequalities yield 
 \[
 c\ge 12r_1-6r_2 >-6r_1>-\dfrac{6}{5}r_1,
 \]
 because $r_1<0$. Hence \eqref{eq:cr_V9_4} holds and this second situation cannot occur. Finally, assume that $(r_1,r_2)\in S_{VI}$. The second inequality in \eqref{eq:cr_V4V9} yields $c\ge 6r_1-4r_2$. By the definition of $S_{VI}$, we have $r_2<\dfrac{9}{5}r_1$. Using the last two inequalities, we again deduce
\[
c\ge 6r_1-4r_2> 6r_1-4 \left(\dfrac{9}{5}r_1\right)=-\dfrac{6}{5}r_1,
\]
 showing again that  \eqref{eq:cr_V9_4} holds. Together, we must have $V^4\cap V^9=\emptyset$.

$i=5.$ Assume that $(r_1,r_2,c)\in V^5$. By the definition of $V^5$, we have $r_1,r_2>0$ and $r_2>\dfrac{9}{5}r_1$, which trivially leads to  $V^6\cap V^5=\emptyset$ and $V^7\cap V^5=\emptyset$. We now show $V^8\cap V^5=\emptyset$. Assume that $(r_1,r_2,c)\in V^8\cap V^5$. By the definitions of $V^5$ and $V^8$, we obtain $r_2>3r_1$ and the following two simultaneous inequalities
\begin{equation}\label{eq:cr_V5V8}
c\le 6r_1\quad \hbox{ and }\quad c\ge 2r_2.
\end{equation}
Using $r_2>3r_1$, the second inequality in \eqref{eq:cr_V5V8} yields that $c>6r_1$, which leads to a contradiction to the first inequality in \eqref{eq:cr_V5V8}. Hence $V^8\cap V^5=\emptyset$. Next, we show that $V^5\cap V^9=\emptyset$. Assume that $(r_1,r_2,c)\in V^5\cap V^9$. By the definition of $V^5$ we have $(r_1,r_2)\in S_{II}\cup S_{III}$, giving rise to two possible situations. Assume first that $(r_1,r_2)\in S_{II}$. The definitions of $V^5$, $S_{II}$ and $V^9$ imply the following two simultaneous inequalities.
\begin{equation}\label{eq:cr_V5V9_a}
c<-6r_1+4r_2\quad \hbox{ and }\quad c\ge\max\{|12r_1-6r_2|,|-6r_1+4r_2|\}\ge -6r_1+4r_2,
\end{equation}
and it is clear that the above two inequalities contradict each other. Hence, this situation cannot occur. Assume now that $(r_1,r_2)\in S_{III}$. Similarly, we derive the following simultaneous inequalities.
\begin{equation}\label{eq:cr_V5V9_b}
 c\le 6r_1\quad \hbox{ and }\quad c\ge\max\{|12r_1-6r_2|,|-6r_1+4r_2|\}\ge -12r_1+6r_2,
\end{equation}
By the definitions of $V^5$ and $V^9$, we also obtain $r_2>3r_1$. Combining the latter inequality with  the rightmost inequality in \eqref{eq:cr_V5V9_b} yields $c> 6r_1$, contradicting the first inequality in \eqref{eq:cr_V5V9_b}. Consequently, $V^9\cap V^5=\emptyset$.

$i=6.$ Assume that $(r_1,r_2,c)\in V^6$. By the definition of $V^6$, we have $r_1,r_2<0$ and $r_2<\dfrac{9}{5}r_1$, which trivially leads to $V^6\cap V^8=\emptyset$. We now show $V^7\cap V^6=\emptyset$. Assume that $(r_1,r_2,c)\in V^7\cap V^6$. By the definitions of $V^6$ and $V^7$, we obtain $r_2<3r_1$ and the following two simultaneous inequalities.
\begin{equation}\label{eq:cr_V6V7}
 c\le -6r_1\quad \hbox{ and }\quad c\ge -2r_2.
\end{equation}
Using $r_2<3r_1$, the second inequality in \eqref{eq:cr_V6V7} yields that $c>-6r_1$, which is a contradiction to the first inequality in \eqref{eq:cr_V5V8}. Hence $V^6\cap V^7=\emptyset$. Next, we show that $V^6\cap V^9=\emptyset$. Assume that $(r_1,r_2,c)\in V^6\cap V^9$. By the definition of $V^6$ we have $(r_1,r_2)\in S_{VI}\cup S_{VII}$, giving rise to two possible situations. Assume first that $(r_1,r_2)\in S_{VI}$. The definitions of $V^6$ and $V^9$ imply the following two simultaneous inequalities.
\begin{equation}\label{eq:cr_V6V9_a}
 c< 6r_1-4r_2\quad \hbox{ and }\quad c\ge\max\{|12r_1-6r_2|,|-6r_1+4r_2|\}\ge -6r_1+4r_2.
\end{equation}
The two inequalities above trivially contradict each other. Hence, this situation cannot occur. Assume now that $(r_1,r_2)\in S_{VII}$, which means that $r_2<3r_1$. Since $(r_1,r_2,c)\in V^6\cap V^9$, we derive the following simultaneous inequalities.
\begin{equation}\label{eq:cr_V6V9_b}
 c\le -6r_1\quad \hbox{ and }\quad c\ge\max\{|12r_1-6r_2|,|-6r_1+4r_2|\}\ge 12r_1-6r_2.
\end{equation}
Combining the rightmost inequality above with the fact that $r_2<3r_1$ yields $c> -6r_1$, contradicting the first inequality in \eqref{eq:cr_V6V9_b}, and consequently, $V^9\cap V^6=\emptyset$.

$i=7.$ From the definition of $V^7$ and $V^8$, we trivially have $V^7\cap V^8=\emptyset$. Assume that $(r_1,r_2,c)\in V^7\cap V^9$. By the definition of $V^9$, we have $c\ge\max\{|12r_1-6r_2|,|-6r_1+4r_2|\}\ge 12r_1-6r_2$. Since $V^7$ requires $c<12r_1-6r_2$, we obtain $V^9\cap V^7=\emptyset$.

$i=8.$ Now we show that $V^9\cap V^8=\emptyset$. From the definition of $V^8$, we have $c<-12r_1+6r_2$, which trivially contradicts $c\ge\max\{|12r_1-6r_2|,|-6r_1+4r_2|\}$ given by $V^9$. Hence, we obtain $V^9\cap V^8=\emptyset$.

Together, we have established the claim in Lemma~\ref{lm:crSet_disjoint}.
\end{proof}

\textbf{\hypertarget{prf:th:dualCarUnconstrIFF}{Proof of Theorem~\ref{th:dualCarUnconstrIFF}}}

\begin{proof}
The proof is structured in three steps. Before proceeding with the proof, we briefly explain here each of the steps.
In Step 0, we list some general facts that will be used for each of the cases.
In Step 1, we show for each $i=1,\ldots,9$, that $(iA)\rightarrow(iB)$ and that $(iA)\rightarrow(iC)$. Namely, we will show that the sign vector $({\rm sign}(h_1(u_c)),{\rm sign}(h_2(u_c)))$ determines the control ${u^{(i)}_c}$ minimizing the Lagrangian, as well as the dual function $q^{(i)}(c)$ and the inclusion  $(r_1,r_2,c)\in V^i$.  In Step 2, we will show $(iB)\rightarrow(iA)$ and $(iC)\rightarrow(iA)$.  In this way, we will obtain the equivalence of $(iA), (iB)$ and $(iC)$ for each $i=1,\ldots,9$.
We proceed now with the proof. \\
{\bf Step 0}.
The facts \eqref{eq:ab_alpha} to \eqref{eq:phi_P1}, which we list next, will be used in {\bf Step 1}. Using \eqref{eq:minimizernCarUnconstr} in Lemma~\ref{lem:subgrad_p1}(ii), we have that $u_c(t)=c\alpha_1(u_c) t- c(\alpha_1(u_c)+\alpha_2(u_c))$. To simplify notation, define
\begin{equation}\label{eq:ab_alpha}
    a:=c\alpha_1(u_c), \quad b:=-c\alpha_1(u_c)-c\alpha_2(u_c).
\end{equation}
Hence we can re-write that $u_c(\cdot)$ as
\begin{equation}\label{eq:uabUnconstr}
    u_c(t)=at+b.
\end{equation}
 Substituting these expressions for $u_c$ into $h_1$ and $h_2$, and solving the resulting integrals, we obtain
\begin{equation}\label{eq:h1_ab}
    h_1(u_c) = r_1+ \int_0^1\left[\int_0^t(a\tau+b)d\tau\right]dt
    =\frac{a+3b+6r_1}{6}.
\end{equation}
Similarly,
\begin{equation}\label{eq:h2_ab}
h_2(u_c)= r_2+\int_0^1(at+b)dt=\frac{a+2b+2r_2}{2}.
\end{equation}
Using \eqref{eq:ab_alpha}, we can substitute $a$ and $b$ into \eqref{eq:h1_ab} and \eqref{eq:h2_ab} to obtain the following.
\begin{equation}\label{eq:h_alpha}
\hspace{-3mm}h_1(u_c) = \frac{-2c\alpha_1(u_c)-3c\alpha_2(u_c)+6r_1}{6},\,
h_2(u_c)= \frac{-c\alpha_1(u_c)-2c\alpha_2(u_c)+2r_2}{2}.
\end{equation}
Using the definition of $\varphi$ we calculate $\varphi(u_c)$.
\begin{equation}\label{eq:phi_P1}
\varphi(u_c)=\frac{1}{2}\int_0^1(at+b)^2 dt=\frac{1}{6}a^2 + \frac{1}{2}ab+ \frac{1}{2}b^2.
\end{equation}

{\bf Step 1} We now show that $(iA)$ implies $(iB)$ and $(iC)$ for Cases 1--9. To achieve this, we consider next each of the nine possibilities for ${\rm sign\,}(h_i(u_c))\in \{0,1,-1\}$  for $i=1,2$.

Case 1. $h_1(u_c)>0$ and $h_2(u_c)>0$. 
This fact, combined with \eqref{eq:alfa} yields $\alpha_1(u_c)=\alpha_2(u_c)=1$. By \eqref{eq:ab_alpha} this implies that
$a=c$ and $b=-2c$. The latter, combined with 
\eqref{eq:uabUnconstr} yields $u_c(t)=:u^{(1)}_c(t)=ct-2c$, establishing $(1A)\rightarrow(1B)$.
Let us show now that $(1A)\rightarrow(1C)$. Note that $(1C)$ has two parts: the formula for the dual function and the inclusion $(r_1,r_2,c)\in V^1$. We start by establishing the formula for the dual function.

By $(1A)$, we have that $\alpha_1(u_c)=\alpha_2(u_c)=1$. Substituting these values into the expressions for $h_1$ and $h_2$ in \eqref{eq:h_alpha} and using the assumption on the signs of $h_1(u_c)$ and $h_2(u_c)$ gives
\begin{equation}\label{eq:h1c_case1}
h_1(u_c)=\frac{-5c+6r_1}{6}>0,
\end{equation}
and
\begin{equation}\label{eq:h2c_case1}
h_2(u_c)=\frac{-3c+2r_2}{2}>0.
\end{equation}
 Let us now compute the dual function $q(c)=\min_u L(u,c)=L(u_c,c)$. By $(1A)$, we already have that $a=c$ and $b=-2c$. Combining this fact with \eqref{eq:phi_P1} yields
\[
\varphi(u_c)=\dfrac{1}{6}c^2 + \dfrac{1}{2}c(-2c)+ \dfrac{1}{2}(-2c)^2=\dfrac{7}{6}c^2.
\]
Hence, the dual function $q(c)$ for this case is
\begin{align*}
q(c)&=L(u_c,c)=\varphi(u_c)+c|h_1(u_c)|+c|h_2(u_c)|\\
&\\
&=\dfrac{7}{6}c^2+\dfrac{(-5c^2+6cr_1)}{6}+\dfrac{(-3c^2+2cr_2)}{2}=-\dfrac{7}{6}c^2+(r_1+r_2)c,
\end{align*}
where we used \eqref{eq:h1c_case1} and \eqref{eq:h2c_case1}. By definition of $q^{(1)}$, we have proved that $(1A)$ implies the first part of $(1C)$. We now show that $(1A)$ implies the second part of $(1C)$. The inequalities in \eqref{eq:h1c_case1} and \eqref{eq:h2c_case1} give $ c<\dfrac{6}{5}r_1$ and $c<\dfrac{2}{3}r_2$ respectively. Hence, we have that 
\begin{equation}\label{eq:cr_case1}
c<\min\left\{\dfrac{6}{5}r_1,\dfrac{2}{3}r_2\right\}.
\end{equation}
Using the fact that $c>0$, \eqref{eq:cr_case1} yields that $r_1,\,r_2>0$, and $0<c< \min\left\{\dfrac{6}{5}r_1,\dfrac{2}{3}r_2\right\}$, which establishes the second part of (1C) (namely that $(r_1,r_2,c)\in V^1$). Hence we have shown that (1A) implies (1C). 

 Case 2. $h_1(u_c)>0$ and $h_2(u_c)<0$. By \eqref{eq:alfa} and \eqref{eq:ab_alpha} we have that $\alpha_1(u_c)=1$, $\alpha_2(u_c)=-1$, $a=c$ and $b=0$. By \eqref{eq:uabUnconstr} this implies that $u_c(t)=u^{(2)}_c(t)=ct$, establishing $(2A)\rightarrow(2B)$. Let us show now that $(2A)\rightarrow(2C)$. By $(2A)$, we have that $\alpha_1(u_c)=1$ and $\alpha_2(u_c)=-1$. Substituting these values into the expressions for $h_1$ and $h_2$ in \eqref{eq:h_alpha} and using the assumption on the signs of $h_1(u_c)$ and $h_2(u_c)$ give
\begin{equation}\label{eq:hc_case2}
h_1(u_c)=\frac{c+6r_1}{6}>0,\,\,h_2(u_c)=\frac{c+2r_2}{2}<0,
\end{equation}
We now compute the dual function $q(c)=L(u_c,c)$. From \eqref{eq:phi_P1} and the fact that $a=c$ and $b=0$, we have that
\begin{align*}
q(c)&=L(u_c,c)=\varphi(u_c)+c|h_1(u_c)|+c|h_2(u_c)|\\[2mm]
&=\frac{1}{6}c^2+\frac{c^2+6cr_1}{6}+\frac{-c^2-2cr_2}{2}= -\frac{1}{6}c^2+(r_1-r_2)c,
\end{align*}
where we used \eqref{eq:hc_case2} in the third equality. By definition of $q^{(2)}$, we have proved that $(2A)$ implies the first part of $(2C)$. We now proceed to show that $(2A)$ implies the second statement in $(2C)$ (namely that $(r_1,r_2,c)\in V^2$). Equation~\eqref{eq:hc_case2}  implies that 
\begin{equation}\label{eq:cr_case2}
-6r_1<c<-2r_2.
\end{equation}
Equation \eqref{eq:cr_case2} and the fact $c>0$, imply that $r_2<3r_1$ and $r_2<0$, which is equivalent to the second part of $(2C)$. We have thus proved that $(2A)$ implies $(2C)$. 

 Case 3. $h_1(u_c)<0$ and $h_2(u_c)>0$. In this case $\alpha_1(u_c)=-1$, $\alpha_2(u_c)=1$, so $a=-c$ and $b=0$. By \eqref{eq:uabUnconstr} this implies that $u_c(t)=u^{(3)}_c(t)=-ct$, establishing $(3A)\rightarrow(3B)$. Now, let us show that $(3A)\rightarrow(3C)$. We start by computing the dual function. By $(3A)$, we have that $\alpha_1(u_c)=-1$ and $\alpha_2(u_c)=1$. Substituting these values into the expressions for $h_1$ and $h_2$ in \eqref{eq:h_alpha} and using the assumption on the signs of $h_1(u_c)$ and $h_2(u_c)$ give
  \begin{equation}\label{eq:hc_case3}
h_1(u_c)=\frac{-c+6r_1}{6}<0,\,\,h_2(u_c)=\frac{-c+2r_2}{2}>0.
\end{equation}
 From \eqref{eq:phi_P1} and the fact that $a=-c$ and $b=0$, we have
\begin{align*}
q(c)&=L(u_c,c)=\varphi(u_c)+c|h_1(u_c)|+c|h_2(u_c)|=\frac{1}{6}c^2+\frac{c^2-6cr_1}{6}+\frac{-c^2+2cr_2}{2}\\[2mm]
&=-\frac{1}{6}c^2+(-r_1+r_2)c,
\end{align*}
where we used \eqref{eq:hc_case3} in the third equality. Hence we have established the first part of $(3C)$. Let us now show that the second part of $(3C)$ also follows from $(3A)$. Indeed, by \eqref{eq:hc_case3} we deduce that 
\begin{equation}\label{eq:cr_case3}
6r_1<c<2r_2.
\end{equation}
Since $c>0$, we obtain that $r_2>3r_1$ and $r_2>0$, which is equivalent to the second part of $(3C)$.

 Case 4. $h_1(u_c)<0$ and $h_2(u_c)<0$. In this case $\alpha_1(u_c)=\alpha_2(u_c)=-1$, so $a=-c$ and $b=2c$. By \eqref{eq:uabUnconstr} this implies that $u_c(t)=u^{(4)}_c(t)=-ct+2c$, establishing $(4A)\rightarrow(4B)$. Let us show now that $(4A)\rightarrow(4C)$. We start by computing the dual function. By $(4A)$, we have that $\alpha_1(u_c)=\alpha_2(u_c)=-1$. Substituting these values into the expressions for $h_1$ and $h_2$ in \eqref{eq:h_alpha} and using the assumption on the signs of $h_1(u_c)$ and $h_2(u_c)$ give
 \begin{equation}\label{eq:hc_case4}
h_1(u_c)=\frac{5c+6r_1}{6}<0,\,\,h_2(u_c)=\frac{3c+2r_2}{2}<0.
\end{equation}
 From \eqref{eq:phi_P1} and the fact that $a=-c$ and $b=2c$ we obtain
\begin{align*}
q(c)&=\varphi(u_c)+c|h_1(u_c)|+c|h_2(u_c)|\\[2mm]
&=\frac{7}{6}c^2+\frac{-5c^2-6cr_1}{6}+\frac{-3c^2-2cr_2}{2}=-\frac{7}{6}c^2-(r_1+r_2)c,
\end{align*}
so the first part of $(4C)$ is established. We now show that $(4A)$ implies the second part of $(4C)$. From the two inequalities in \eqref{eq:hc_case4}, we obtain that 
\begin{equation}\label{eq:cr_case4}
c<\min\left\{-\dfrac{6}{5}r_1,-\dfrac{2}{3}r_2\right\}.
\end{equation}
Since $c>0$, we have that $r_1,\,r_2<0$, and $0<c<\min\left\{-\dfrac{6}{5}r_1,-\dfrac{2}{3}r_2\right\}$. Hence we showed that $(4A)$ also implies the second part of $(4C)$.

Case 5. $h_1(u_c)=0$ and $h_2(u_c)>0$. In this case, $\alpha_2(u_c)=1$, and $\alpha_1(u_c)\in [-1,1]$ (note that $\alpha_1(u_c)$ can be any element in $[-1,1]$). Using also \eqref{eq:ab_alpha} we can write
\begin{equation}\label{eq:ab_case5}
a=c\alpha_1(u_c), \quad b=-c\alpha_1(u_c)-c,
\end{equation}
which, combined with \eqref{eq:h_alpha} gives
\begin{equation}\label{eq:h1_case5}
    h_1(u_c)=\frac{-2c\alpha_1(u_c)-3c+6r_1}{6}=0.
\end{equation}
Altogether, we deduce that  
\begin{equation}\label{eq:c_alpha1_case5}
     a=c\alpha_1(u_c)=\dfrac{-3c+6r_1}{2}.
\end{equation}
The expression above and \eqref{eq:ab_case5} yield $u_c(t)=u^{(5)}_c(t)=\dfrac{-3c+6r_1}{2}t+\dfrac{c-6r_1}{2}$, establishing $(5A)\rightarrow(5B)$.
To show the first part of $(5C)$, we compute $h_2(u_c)$. By \eqref{eq:h_alpha} we have
\begin{equation}\label{eq:h2_case6}
h_2(u_c)=\dfrac{-c\alpha_1(u_c)-2c+2r_2}{2}=\dfrac{-c-6r_1+4r_2}{4}>0,
\end{equation}
where we used \eqref{eq:c_alpha1_case5} in the second equality. To show that $q(c)=q^{(5)}(c)$, we compute first $\varphi(u_c)$ using \eqref{eq:ab_case5} and \eqref{eq:phi_P1} as follows,
\begin{align*}
\varphi(u_c)&    =\frac{1}{6}(c\alpha_1(u_c))^2 + \frac{1}{2}(c\alpha_1(u_c))(-c\alpha_1(u_c)-c)+ \frac{1}{2}(-c\alpha_1(u_c)-c)^2\\[2mm]
  & =\frac{1}{6}c^2\alpha_1^2(u_c) + \frac{1}{2}c^2\alpha_1(u_c) + \frac{1}{2}c^2\\[2mm]
  & =\frac{1}{6} \left(\frac{-3c+6r_1}{2}\right)^2
    + \frac{1}{2} \frac{c(-3c+6r_1)}{2} + \frac{1}{2} c^2 =\frac{1}{8}c^2+\frac{3}{2}r_1^2,
\end{align*}
where we are using \eqref{eq:c_alpha1_case5} in the third equality. Now we are ready to compute $q(c)$.
By assumption we have that $c|h_1(u_c)|=0$, and we also have that
\[
c|h_2(u_c)|=\frac{-c^2-6cr_1+4cr_2}{4}.
\]
Altogether, we can write
\[
\begin{array}{rcl}
 q(c)=L(u_c,c)    &  = &\varphi(u_c)+c|h_2(u_c)|=\dfrac{1}{8}c^2+\dfrac{3}{2}r_1^2 +\dfrac{-c^2-6cr_1+4cr_2}{4}\\[2mm]
     &= &-\dfrac{1}{8}c^2+(-\dfrac{3}{2}r_1+r_2)c+\dfrac{3}{2}r_1^2=q^{(5)}(c),
\end{array}
  \]
 which establishes the first part of $(5C)$. We now proceed to prove that $(5A)$ implies the second statement in $(5C)$. Since $\alpha_1(u_c)\in[-1,1]$, $3+2\alpha_1(u_c)\ge 1> 0$. Equation~\eqref{eq:c_alpha1_case5} yields
\begin{equation}\label{eq:cr_case5_1}
c= \dfrac{6r_1}{3+2\alpha_1(u_c)}.
\end{equation}
From \eqref{eq:h2_case6} we derive
that 
\begin{equation}\label{eq:cr_case5_2}
c \in(0,-6r_1+4r_2). \quad \mbox{Consequently,}\quad r_2>\dfrac{3}{2}r_1.
\end{equation}
Using the fact that $\alpha_1(u_c)\in[-1,1]$, we obtain the range of $c$ from Equation \eqref{eq:cr_case5_1}. That is, we deduce that
\begin{equation}\label{eq:cr_case5_3}
c \in\left[\dfrac{6}{5}r_1,6r_1\right] \quad\mbox{and}\quad r_1>0.
\end{equation}
Since $r_1$, $r_2$ satisfy both \eqref{eq:cr_case5_3} and \eqref{eq:cr_case5_2}, we obtain $r_1,r_2>0$. Noting that $c$ satisfies both \eqref{eq:cr_case5_3} and \eqref{eq:cr_case5_2}, we have 
\begin{equation}\label{eq:cr_case5_4}
    c\in (0,-6r_1+4r_2)\cap \left[\dfrac{6}{5}r_1,6r_1\right].
\end{equation}
For the intersection in \eqref{eq:cr_case5_4} to be nonempty, one must have $\dfrac{6}{5}r_1<-6r_1+4r_2$, which yields $r_2>\dfrac{9}{5}r_1$. We now consider the lower bound of $c$ from the intersection in \eqref{eq:cr_case5_4}. Since $\dfrac{6}{5}r_1>0$, we have  $c\ge\dfrac{6}{5}r_1$. For the upper bound of $c$ from \eqref{eq:cr_case5_4}, we have two possibilities: 
\[
 \mbox{Either }-6r_1+4r_2\le 6r_1, \,\, \mbox{yielding} \,\, r_2\le 3r_1; \hbox{ or } \,
  -6r_1+4r_2> 6r_1, \,\, \mbox{yielding}\,\, r_2> 3r_1.
\]
Together, these two complementary cases and \eqref{eq:cr_case5_4} yield
\[
\dfrac{6}{5}r_1\le c < -6r_1+4r_2, \,\, \mbox{for} \,\,  \dfrac{9}{5}r_1<r_2\le 3r_1;\hbox{ or } \quad
\dfrac{6}{5}r_1\le c \le 6r_1, \,\, \mbox{for} \,\,  r_2> 3r_1,
\]
where we also include the  lower bound in the result. By definition of $V^5$, the above expressions show that $(5A)$ implies the second part of $(5C)$.
    
Case 6.  $h_1(u_c)=0$ and $h_2(u_c)<0$.
In this case, $\alpha_2(u_c)=-1$, and
\begin{equation}\label{eq:ab_case6}
a=c\alpha_1(u_c), \quad b=-c\alpha_1(u_c)+c.
\end{equation}
Proceeding as in the previous case, we obtain
\begin{equation}\label{eq:hc_case6}
h_1(u_c)=\frac{-2c\alpha_1(u_c)+3c+6r_1}{6}=0, \quad h_2(u_c)=\frac{c-6r_1+4r_2}{4}<0.
\end{equation}
Using the equality above involving $h_1$, we can write 
 \begin{equation}\label{eq:cr_case6_1}
 c=\dfrac{6r_1}{-3+2\alpha_1(u_c)} \quad \mbox{and} \quad a= c\alpha_1(u_c)=\dfrac{6r_1+3c}{2}.
 \end{equation} 
The second equality above, together with \eqref{eq:ab_case6} gives 
\[
u_c(t)=at+b= \dfrac{6r_1+3c}{2} t -\dfrac{(6r_1+c)}{2}=u^{(6)}_c(t),
 \]
 establishing $(6B)$. We proceed now to show $(6A)\rightarrow(6C)$. To derive $q(c)$, we compute first $\varphi(u_c)$, using \eqref{eq:ab_case6} and \eqref{eq:phi_P1} as follows,
\begin{align*}
\varphi(u_c)&=\frac{1}{6}(c\alpha_1(u_c))^2 + \frac{1}{2}(c\alpha_1(u_c))(-c\alpha_1(u_c)+c)+ \frac{1}{2}(-c\alpha_1(u_c)+c)^2\\
      &=\frac{1}{6}c^2\alpha_1^2(u_c) - \frac{1}{2}c^2\alpha_1(u_c) + \frac{1}{2}c^2
    =\frac{1}{8}c^2+\frac{3}{2}r_1^2,
\end{align*}
where \eqref{eq:cr_case6_1} is used in the last equation. We now proceed to the computation of $q(c)$. The assumption on $h_1$ gives
$c|h_1(u_c)|=0$, and by \eqref{eq:hc_case6}, $c|h_2(u_c)|=\dfrac{-c^2+6cr_1-4cr_2}{4}$. Altogether, we can write,
\[
\begin{array}{rcl}
  q(c)=L(u_c,c)    &  = &\varphi(u_c)+c|h_2(u_c)|=\dfrac{1}{8}c^2+\dfrac{3}{2}r_1^2 +\dfrac{-c^2+6cr_1+4cr_2}{4}\\
 &&\\
     &= &-\dfrac{1}{8}c^2+\left(\dfrac{3}{2}r_1-r_2\right)c+\dfrac{3}{2}r_1^2=q^{(6)}(c),
\end{array}
  \]
which establishes the first part of $(6C)$. We now proceed to show that $(6A)$ implies the second part of  $(6C)$ (namely that $(r_1,r_2,c)\in V^6$). The expression for $h_2(u_c)$ in \eqref{eq:hc_case6} yields 
\begin{equation}\label{eq:cr_case6_2}
     c\in(0,6r_1-4r_2).
\end{equation}
Using the fact that $\alpha_1(u_c)\in[-1,1]$ and the first equation in \eqref{eq:cr_case6_1} we obtain
 \begin{equation}\label{eq:cr_case6_3}
     r_1<0,\,\mbox{and}\, c\in\left[-\frac{6}{5}r_1, -6r_1\right].
 \end{equation}
We derive that $r_1,r_2<0$ from $6r_1-4r_2>0$ in \eqref{eq:cr_case6_2} and $r_1<0$ in \eqref{eq:cr_case6_3}. For the range of $c$, using \eqref{eq:cr_case6_2} and \eqref{eq:cr_case6_3}, we further obtain the following.
\begin{equation}\label{eq:cr_case6_4}
     c\in(0,6r_1-4r_2)\cap\left[-\frac{6}{5}r_1, -6r_1\right] .
\end{equation}
The non-emptiness of the intersection in \eqref{eq:cr_case6_4} requires $-\dfrac{6}{5}r_1<6r_1-4r_2$, which leads $r_2<\dfrac{9}{5}r_1$. Since $-\dfrac{6}{5}r_1>0$, we have from the lower bound of the intersection that $c\ge-\dfrac{6}{5}r_1$. As in the previous case, for the upper bound of the intersection we consider two possibilities.
\[
\mbox{Either }6r_1-4r_2\le -6r_1, \,\, \mbox{yielding} \,\, r_2\ge 3r_1; \hbox{ or } \,
  6r_1-4r_2> -6r_1, \,\, \mbox{yielding}\,\, r_2< 3r_1.
\]
In the rightmost case, we deduce that  $-\dfrac{6}{5}r_1 \le c < 6r_1-4r_2$ for $3r_1\le r_2<\dfrac{9}{5}r_1$. In the leftmost case we obtain $-\dfrac{6}{5}r_1 \le c \le -6r_1$ for $r_2< 3r_1$. By definition of $V^6$, we have established that $(6A)$ implies $(6C)$.

Case 7. $h_1(u_c)>0$ and $h_2(u_c)=0$. In this case, $-1\le\alpha_2(u_c)\le1$, $\alpha_1(u_c)=1$ and
\begin{equation}\label{eq:ab_case7}
    a=c, \quad b=-c-c\alpha_2(u_c).
\end{equation}
From \eqref{eq:h_alpha} and the fact that $h_2(u_c)=0$, we obtain
\begin{equation} \label{eq:h2_case7}
h_2(u_c)=\frac{-c-2c\alpha_2(u_c)+2r_2}{2}=0,
\end{equation}
which yields
\begin{equation} \label{eq:cr_case7_1}
    c(1+2\alpha_2(u_c))=2r_2 \quad \mbox{and} \quad
    c\alpha_2(u_c)=\dfrac{-c+2r_2}{2}.
\end{equation}
From \eqref{eq:ab_case7} and the second equality above, we deduce that $a=c$ and $b=-(c/2)-r_2$ which trivially yields 
\[
u_c(t)=at+b= c t -\dfrac{(2r_2+c)}{2}=u^{(7)}_c(t),
 \]
 establishing $(7B)$. We proceed now to show $(7A)\rightarrow(7C)$. We start with the computation of the dual function. Using \eqref{eq:h_alpha} and the expression for $c\alpha_2(u_c)$ in \eqref{eq:cr_case7_1} we obtain
\begin{equation} \label{eq:h1_case7}
  h_1(u_c)=\dfrac{-c+12r_1-6r_2}{12}>0.
\end{equation}
We compute $q(c)$ by deriving $\varphi(u_c)$ as follows. Using \eqref{eq:phi_P1}, \eqref{eq:ab_case7} and \eqref{eq:cr_case7_1}, we have
 \begin{equation} \label{eq:phi_case7}
  \begin{array}{rcl}
     \varphi(u_c) & = & \dfrac{1}{6}a^2 + \dfrac{1}{2}ab+ \dfrac{1}{2}b^2=\dfrac{1}{6}c^2 + \dfrac{1}{2}c(-c-c\alpha_2(u_c))+ \dfrac{1}{2}(-c-c\alpha_2(u_c))^2\\
     & &\\
     & = & \dfrac{1}{24}c^2+\dfrac{1}{2}r_2^2.
     \end{array}
 \end{equation}
Expression \eqref{eq:h1_case7} and the assumption on $h_2(u_c)$ give $c|h_1(u_c)|=-\dfrac{1}{12}c^2+c(r_1-\dfrac{1}{2}r_2)$ and $c|h_2(u_c)|=0$, respectively. Hence, we can write
 \[
 q(c)=L(u_c,c)=\varphi(u_c)+c|h_1(u_c)|=-\frac{1}{24}c^2+\left(r_1-\frac{1}{2}r_2\right)c+\frac{1}{2}r_2^2,
 \]
thus establishing the first part of $(7C)$. We now show that  $(7A)$ implies the second part of $(7C)$. From \eqref{eq:h1_case7}, we obtain 
\begin{equation}\label{eq:cr_case7_2}
      0<c<12r_1-6r_2.
 \end{equation}
We now compute the range of $c$ from the first equation in \eqref{eq:cr_case7_1}. We proceed with our calculation by considering three complementary cases according to the possible values of $\alpha_2(u_c)$. When $\alpha_2(u_c)\neq-\dfrac{1}{2}$, \eqref{eq:cr_case7_1} yields
\begin{equation}\label{cr_case7_3}
    c=\dfrac{2r_2}{1+2\alpha_2(u_c)}.
\end{equation}
If $-1\le\alpha_2(u_c)<-\dfrac{1}{2}$, then $-1\le 1+2\alpha_2(u_c)<0$. Since $c>0$, \eqref{cr_case7_3} yields that $r_2<0$. Using the lower bound of the denominator in \eqref{cr_case7_3} gives $c\ge-2r_2$. Using now the upper bound for $c$ in \eqref{eq:cr_case7_2} we have that $c<12r_1-6r_2$. Combining the last two inequalities for $c$ yield $-2r_2<12r_1-6r_2$. Therefore, $r_2<3r_1$ and $-2r_2\le c<12r_1-6r_2$, which establishes the first condition in the definition of $V^7$. We proceed now to establish the two remaining conditions.

If $-\dfrac{1}{2}<\alpha_2(u_c)\le 1$, then $0<1+2\alpha_2(u_c)\le 3$. Since $c>0$, \eqref{cr_case7_3} yields that $r_2>0$ and $c\ge\dfrac{2}{3}r_2$. Since $c$ also satisfies \eqref{eq:cr_case7_2}, we have $\dfrac{2}{3}r_2<12r_1-6r_2$, which leads to $r_2<\dfrac{9}{5}r_1$. Altogether, we obtain that $r_1>0$ from $r_2>0$; and $\dfrac{2}{3}r_2\le c<12r_1-6r_2$. Hence we have shown the second condition in the definition of $V^7$.

When $\alpha_2(u_c)=-\dfrac{1}{2}$, using the first equation in \eqref{eq:cr_case7_1}, we derive $r_2=0$. From \eqref{eq:cr_case7_2}, we further have $r_1>0$. Thus we proved the last condition in the definition of $V^7$. Altogether, we have shown that $(7A)\rightarrow(7C)$.

Case 8.      $h_1(u_c)<0$ and $h_2(u_c)=0$. In this case, $-1\le\alpha_2(u_c)\le1$, $\alpha_1(u_c)=-1$, and
      \begin{equation}\label{eq:ab_case8}
          a=-c, \quad b=c-c\alpha_2(u_c).
      \end{equation}
Using the assumption on $h_2(u_c)$, as well as \eqref{eq:h_alpha}, we obtain
\begin{equation}\label{eq:h2_case8a}
h_2(u_c)=\frac{c-2c\alpha_2(u_c)+2r_2}{2}=0,
\end{equation}
which further gives us
  \begin{equation} \label{eq:cr_case8_1}
      c(-1+2\alpha_2(u_c))=2r_2 \quad \mbox{and} \quad c\alpha_2(u_c)=\dfrac{c+2r_2}{2}.
 \end{equation}
The second equality above, together with \eqref{eq:ab_case8} yield $a=-c$ and $b=(c/2)-r_2$. Hence,
\[
u_c(t)=at+b=-ct +(c/2)-r_2= u^{(8)}_c(t),
 \]
 establishing $(8B)$.
We now show that $(8A)\rightarrow (8C)$. 
From $h_1(u_c)<0$ and \eqref{eq:h_alpha}, we have
\begin{equation}\label{eq:h2_case8}
h_1(u_c)=\frac{c+12r_1-6r_2}{12}<0. 
\end{equation}
 We now proceed to the computation of $\varphi(u_c)$ in the purpose of deriving $q(c)$. The expressions for $\varphi(u_c)$ in \eqref{eq:phi_P1} and for $c\alpha_2(u_c)$ in \eqref{eq:cr_case8_1} result in
  \begin{equation}\label{eq:L3C8}
      \varphi(u_c)=\frac{1}{6}(-c)^2 + \frac{1}{2}(-c)(c-c\alpha_2(u_c))+ \frac{1}{2}(c-c\alpha_2(u_c))^2
 = \frac{1}{24}c^2+\frac{1}{2}r_2^2.  
  \end{equation}
 Using the facts that $c|h_2(u_c)|=0$ and $c|h_1(u_c)|=-\dfrac{1}{12}c^2+c\left(\dfrac{1}{2}r_2-r_1\right)$ and \eqref{eq:L3C8}, obtain the following.
 \[
 q(c)=L(u_c,c)=\varphi(u_c)+c|h_1(u_c)|=-\frac{1}{24}c^2+\left(\frac{1}{2}r_2-r_1\right)c+\frac{1}{2}r_2^2=q^{(8)}(c),
 \]
 which establishes the first part of $(8C)$. We now proceed to show the second part of $(8C)$. The inequality in \eqref{eq:h2_case8} gives us
  \begin{equation} \label{eq:cr_case8_2}
      0<c<-12r_1+6r_2.
 \end{equation}
  In a similar manner as in Case 7, and according to three complementary possible values for $\alpha_2(u_c)$, we will obtain $(8C)$. When $\alpha_2(u_c)\neq\dfrac{1}{2}$, we derive from \eqref{eq:cr_case8_1} that
   \begin{equation} \label{eq:cr_case8_3}
      c=\dfrac{2r_2}{-1+2\alpha_2(u_c)}.
 \end{equation}
If $\dfrac{1}{2}<\alpha_2(u_c)\le1$, then $1\ge -1+2\alpha_2(u_c)>0$. Using $c>0$, \eqref{eq:cr_case8_3} yields $r_2>0$ and $c\ge2r_2$, which gives a lower bound for $c$. Combining with the upper bound of $c$ in \eqref{eq:cr_case8_2}, we obtain $2r_2<-12r_1+6r_2$, leading to $r_2>3r_1$ and $2r_2\le c<-12r_1+6r_2$. Hence we have shown the first condition in the definition of $V^8$.
 
If $-1\le\alpha_2(u_c)<\dfrac{1}{2}$, then $-3\le -1+2\alpha_2(u_c)<0$. Since $c>0$, \eqref{eq:cr_case8_3} yields $r_2<0$ and $c\ge-\dfrac{2}{3}r_2$, which gives a lower bound for $c$. Combining with the upper bound of $c$ in \eqref{eq:cr_case8_2}, we obtain that $-\dfrac{2}{3}r_2<-12r_1+6r_2$, which gives $r_2>\dfrac{9}{5}r_1$, $r_1<0$ and $-\dfrac{2}{3}r_2 \le c<-12r_1+6r_2$. Hence the second condition in the definition of $V^8$ is established.

When $\alpha_2(u_c)=\dfrac{1}{2}$, the first equation in \eqref{eq:cr_case8_1} gives $r_2=0$. Combining this fact with \eqref{eq:cr_case8_2}, we have $r_1<0$. Hence we have established the last condition in the definition of $V^8$. Altogether, we have proved that $(8A)\rightarrow(8C)$.

 Case 9.  $h_1(u_c)=0$ and $h_2(u_c)=0$. This assumption, together with \eqref{eq:h1_ab} and \eqref{eq:h2_ab} gives
 \begin{equation}\label{eq:h1_case9}
h_1(u_c)=\frac{a+3b+6r_1}{6}=0,\,\,h_2(u_c)=\frac{a+2b+2r_2}{2}=0.
\end{equation}
Hence,
\begin{equation}\label{eq:a&b_case9}
    a=12r_1-6r_2,\, \, b=-6r_1+2r_2.
\end{equation}
This readily gives $u_c(t)=u^{(9)}_c(t)$, establishing $(9B)$. Next, we compute $q$. Combining \eqref{eq:a&b_case9} with \eqref{eq:phi_P1} gives
\begin{align*}
q(c) & = L(u_c,c)=\varphi(u_c)\\
&\\
 & =\frac{1}{6}(12r_1-6r_2)^2 + \frac{1}{2}(12r_1-6r_2)(-6r_1+2r_2)+ \frac{1}{2}(-6r_1+2r_2)^2\\
&\\
   &=6r_1^2-6r_1r_2+2r_2^2,
\end{align*}
 where we used the fact that $h_1(u_c)=h_2(u_c)=0$ in the second equality. Note that in this case $q$ is constant, since it does not depend on $c$. Thus $(9A)$ implies the first part of $(9C)$. Let us now show that $(9A)$ implies the second part of $(9C)$. Assumption $(9A)$ gives $a=c\alpha_1(u_c)$ and $b=-c(\alpha_1(u_c)+\alpha_2(u_c))$ with $\alpha_i(u_c)\in [-1,1]$ for $i=1,2$. Since $a=12r_1-6r_2=c\alpha_1(u_c)$ and $\alpha_1(u_c)\in [-1,1]$, we can write
\[
|a|=|12r_1-6r_2 | =|c\alpha_1(u_c)|= c|\alpha_1(u_c)|\le c,
\]
so that $c\ge|12r_1-6r_2 |$. Similarly, elementary calculations give the following.
\[
b=-6r_1+2r_2=-c(\alpha_1(u_c)+\alpha_2(u_c))=-a -c\alpha_2(u_c)= -(12r_1-6 r_2)-c\alpha_2(u_c) ,
\]
so $c\alpha_2(u_c)= -6 r_1+4r_2$. The latter yields
\[
|-6 r_1+4r_2|=|c\alpha_2(u_c)|= c|\alpha_2(u_c)|\le c,
\]
so $c\ge|-6 r_1+4r_2 |$. Together, we deduce that
\begin{equation}\label{eq:cr_case9}
    c\ge\max\{|12r_1-6r_2|,|-6r_1+4r_2|\}.
\end{equation}
 Hence we have shown that $(9A)$ implies the second part of $(9C)$. This completes {\bf Step 1}. That is, we have shown that $(iA)$ implies $(iB)$ and $(iC)$ for $i=1,\ldots,9$.\\ 
 {\bf Step 2}  We will show in this step that $(iB)$ implies $(iA)$, and that $(iC)$ implies $(iA)$ for all $i=1,\ldots ,9$. 
Let us show that $(iB)\rightarrow (iA)$. For every $j\in \{1,\ldots,9\}$, 
denote by $(a^{(j)},b^{(j)})$ the vector of coefficients of the control function $u^{(j)}_c$ computed in Case $j$ by means of \eqref{eq:ab_alpha} and \eqref{eq:uabUnconstr}. By inspecting the proofs of all cases, we can easily check that $(a^{(j)},b^{(j)})\neq (a^{(s)},b^{(s)})$ whenever $j\neq s$, $j,s\in \{1,\ldots,9\}$. Assume that $(iB)$ holds. Since one and only one of the cases $(jA)$ holds, assume that $(kA)$ holds. We will show that $k=i$. Indeed, by the first part of the theorem we have
\[
 (kA) \rightarrow (kB).
\]
Since $(kA)$ holds, we must have $(kB)$ holds. So, $(iB)$ and $(kB)$ hold simultaneously. This implies that
$(a^{(i)},b^{(i)})= (a^{(k)},b^{(k)})$. But the latter equality can only occur if $k=i$. We have thus shown that $(iB)\rightarrow (iA)$.
 
We proceed to show that $(iC)$ implies $(iA)$.  Note again that at least one of the cases $(kA)$ must hold, since they cover all possibilities. So assume that $(iC)$ holds and let $(kA)$ hold for some $k\in \{1,\ldots,9\}$. We will show that $k=i$.
By Step 1, we have
\[
(kA) \rightarrow (kC),
\]
This means that both $(iC)$ and $(kC)$ hold simultaneously. Using the second parts of these statements, we deduce that
\[
(r_1,r_2,c)\in V^i \cap V^k,
\]
which yields $V^i \cap V^k\not=\emptyset$. According to Lemma~\ref{lm:crSet_disjoint}, this can only happen when $i=k$. Thus, we have proved that $(iC)\rightarrow(iA)$. The proof of the equivalence between the three statements is complete.
\end{proof}

\end{document}